\documentclass[12pt,a4paper]{amsart}
\usepackage{graphicx,amssymb,stmaryrd,xspace}
\RequirePackage[raiselinks=false,colorlinks=true,citecolor=blue,urlcolor=blue,linkcolor=blue,bookmarksopen=true,dvips]{hyperref}
\def\fam#1{\widetilde{#1}}
\def\into{\hookrightarrow}

\def\onto{\twoheadrightarrow}
\usepackage[cmtip,all,color,dvips]{xy}
\newdir{ >}{{}*!/-7pt/@{>}}
\newdir{> }{{}*!/10pt/@{>}}
\newdir{>> }{{}*!/10pt/@{>>}}

\newcommand{\dlpullback}[1][dl]{\save*!/#1-1.2pc/#1:(-1,1)@^{|-}\restore}

\newcommand{\drpullback}[1][dr]{\save*!/#1-1.2pc/#1:(-1,1)@^{|-}\restore}
\def\rTo{\to}
\def\lTo{\leftarrow}

\newcommand{\grint}[3]{\int^{#1 \in #2} #3}

\newcommand{\leaves}[1]{#1}
\newcommand{\treebialg}{\mathcal{B}}
\newcommand{\efftreebialg}{\mathcal{B}_{\mathrm{eff}}}

\newcommand{\CK}{\mathcal{H}}
\providecommand{\norm}[1]{\left| {#1}\right|}

\newtheorem{lemma}{Lemma.}[section]
\newtheorem{prop}[lemma]{Proposition.}
\newtheorem{thm}[lemma]{Theorem.}
\newtheorem{cor}[lemma]{Corollary.}
\newtheorem{taller}[lemma]{$\!\!$}
\newenvironment{blanko}[1]%
{\begin{taller}{\normalfont\bfseries  #1}\normalfont}%
{\end{taller}}

\newenvironment{blanko*}[1]%
{%
\begin{list}{\bf {#1} }%
{\setlength{\labelsep}{0mm}\setlength{\leftmargin}{0mm}%
\setlength{\labelwidth}{0mm}\setlength{\listparindent}{\parindent}%
\setlength{\parsep}{\parskip}\setlength{\partopsep}{0mm}}%
\item%
}%
{%
\end{list}%
}

{%
\begin{list}{\em Definition. }%
{\setlength{\labelsep}{0mm}\setlength{\leftmargin}{0mm}%
\setlength{\labelwidth}{0mm}\setlength{\listparindent}{\parindent}%
\setlength{\parsep}{\parskip}\setlength{\partopsep}{0mm}}%
\item%
}%
{%
\end{list}%
}
\newenvironment{dem}%
{%
\begin{list}{\em Proof. }%
{\setlength{\labelsep}{0mm}\setlength{\leftmargin}{0mm}%
\setlength{\labelwidth}{0mm}\setlength{\listparindent}{\parindent}%
\setlength{\parsep}{\parskip}\setlength{\partopsep}{0mm}}%
\item%
}%
{%
\qed\end{list}%
}
\newenvironment{dem*}[1]%
{%
\begin{list}{\em #1 }%
{\setlength{\labelsep}{0mm}\setlength{\leftmargin}{0mm}%
\setlength{\labelwidth}{0mm}\setlength{\listparindent}{\parindent}%
\setlength{\parsep}{\parskip}\setlength{\partopsep}{0mm}}%
\item%
}%
{%
\qed\end{list}%
}

\usepackage{mathrsfs}

\def\Id{\text{Id}}
\def\surj{\mathbf S}
\def\nonempt{\mathbf B}
\def\Bij{\mathbf{Bij}}

\def\tr{\mathbf T}
\def\forest{\mathbf F}

\def\treescuts{\mathbf C}
\def\cut{\operatorname{cut}}
\def\fdb{Fa{\`a} di Bruno\xspace}
\def\fdbsymb{{\mathcal F}}
\def\fdbhopf{{\mathcal H}}%\fdbsymb/(a_1)}

\def\onto{\twoheadrightarrow}

\newcommand{\upperstar}{^{\raisebox{-0.25ex}[0ex][0ex]{\(\ast\)}}}
\newcommand{\lowerstar}{_{\raisebox{-0.33ex}[-0.5ex][0ex]{\(\ast\)}}}
\newcommand{\lowershriek}{_!}

\newcommand{\isopil}{\stackrel{\raisebox{0.1ex}[0ex][0ex]{\(\sim\)}}%
			{\raisebox{-0.15ex}[0.28ex]{\(\rightarrow\)}}}

\providecommand{\fat}[1]{\mathbf{#1}}
\providecommand{\kat}[1]{\text{\textbf{\textsl{#1}}}}

\newcommand{\Set}{\kat{Set}}
\newcommand{\Grpd}{\kat{Grpd}}
\newcommand{\TEmb}{\kat{TEmb}}

\newcommand{\N}{\mathbb{N}}

\newcommand{\Q}{\mathbb{Q}}
\newcommand{\C}{\mathbb{C}}

\newcommand{\name}[1]{\ulcorner #1 \urcorner}

\newcommand{\Aut}{\operatorname{Aut}}

\usepackage{texdraw}
\input{txdtools}

\everytexdraw{%
	\drawdim pt \linewd 0.5 \textref h:C v:C
	\setunitscale 1
}

\newcommand{\Onedot}{
  \bsegment
    \move (0 0) \fcir f:0 r:2.5
  \esegment
}
\newcommand{\smalldot}{
  \bsegment
    \move (0 0) \fcir f:0 r:1.5
  \esegment
}

\newcommand{\trevert}{
\bsegment %\setunitscale{0.8}
\move (1 0) \lvec (6 0) \smalldot \lvec (10 4) \move (6 0) \lvec (10 -4)
\esegment
}

\newcommand{\trekant}{
\bsegment \setunitscale{0.8}
\move (1 0) \lvec (6 0) \smalldot \lvec (16 10)
\move (6 0) \lvec (16 -10)
\move (12 6) \smalldot  \lvec (12 -6) \smalldot
\esegment
}

\newcommand{\tovert}{
\bsegment %\setunitscale{0.8}
\move (1 0) \lvec (6 0) \smalldot \lvec (11 0)
\esegment
}

\newcommand{\tokant}{
\bsegment \setunitscale{0.8}
\move (0 -10) \lvec (0 10)
      \move (0 5) \smalldot
      \move (0 -5) \smalldot
      \move (0 0) \larc r:5 sd:-90 ed:90

\esegment
}

\newcommand{\freeEllipsis}[3]{
    \writeps {
      gsave 
      #3 rotate 
    }
    \lellip rx:#1 ry:#2
    \writeps { 
      grestore
    }
}

\newcommand{\onedot}{
  \bsegment
    \move (0 0) \fcir f:0 r:2
  \esegment
}

  \newcommand{\inlineDotlessTree}{% ha ha !
  \raisebox{-4pt}{
  \begin{texdraw} \linewd 0.5 \bsegment
   \move (0 0) \lvec (0 15) \move (1 0)
 \esegment \end{texdraw} } 
}
  \newcommand{\inlineNullaryTree}{% ha ha !
  \raisebox{-4pt}{
  \begin{texdraw} \linewd 0.5 \bsegment
   \move (0 0) \lvec (0 10) \onedot \move (1 0)
 \esegment \end{texdraw} } 
}

\newcommand{\red}{\writeps{1 0 0 setrgbcolor}}

\usepackage%[usenames,dvipsnames]
{pstricks}
\usepackage{epsfig}

\title[Fa{\`a} di Bruno for Green functions]
{Groupoids and Fa{\`a} di Bruno formulae for Green functions in bialgebras of trees}

\thanks{The first author was partially supported by grants 
MTM2010-15831, MTM2010-20692, SGR1092-2009,  
the second author by 
MTM2009-10359, % Nart
MTM2010-20692, % Castellana
and SGR1092-2009, % Aguad\'e
and the third author by
MTM2010-15831, SGR119-2009.}

\author{Imma G\'alvez-Carrillo}
\address{Departament de Matem\`atica Aplicada III
      \\Universitat Polit\`ecnica de Catalunya
      \\Escola d'Enginyeria de Terrassa 
      \\Carrer Colom 1\\08222 Terrassa (Barcelona)\\Spain}
\email{m.immaculada.galvez@upc.edu}

\author{Joachim Kock}
\address{Departament de Matem\`atiques
       \\Universitat Aut\`onoma de Barcelona
       \\08193 Bellaterra (Barcelona), Spain}
\email{kock@mat.uab.cat}

\author{Andrew Tonks}
\address{STORM, London Metropolitan University\\
166--220 Holloway Road, London N7 8DB, UK}
\email{a.tonks@londonmet.ac.uk}

\date{Friday 2013-11-01}                                         

\begin{document}

\begin{abstract}
  We prove a \fdb formula for the Green function in
  the bialgebra of $P$-trees, for any polynomial endofunctor $P$.  The
  formula appears as relative homotopy cardinality of an
  equivalence of groupoids. 
%   For suitable choices of $P$, the result
%   implies also formulae for Green functions in bialgebras of graphs.
\end{abstract}

\maketitle

\tableofcontents

\section*{Introduction}
%%%%%%%%%%%%%%%%%%%%%%%%%%%%%%%%%%%%%%%%%%%%%%%%%%

This paper is a contribution to the combinatorial understanding of
renormalisation in perturbative quantum field theory.  It can be
seen as part of the general programme, pioneered by Joyal and Baez--Dolan
(and in a sense already by Grothendieck), of gaining insight into
combinatorics, especially regarding symmetries, by upgrading
from finite sets to suitably finite groupoids.  We derive \fdb formulae
in bialgebras of trees by realising them as relative homotopy
cardinalities of equivalences of groupoids.  
An attractive aspect of this approach is that all issues with symmetries are handled
completely transparently by the groupoid formalism, and take care
of themselves throughout the equivalences without appearing in the calculations.  
This is made possible by our novel and consistent use of {\em homotopy sums}.
The general philosophy is that sums weighted by inverses of symmetry 
factors always arise as groupoid cardinalities of homotopy sums.
It is our hope that these kinds of techniques can
be useful more generally in perturbative quantum field theory,
and related areas.

Our starting point is the seminal work of van Suijlekom on Hopf algebras and
renormalisation of gauge field theories \cite{vanSuijlekom:0610},
\cite{vanSuijlekom:0801}, \cite{vanSuijlekom:0807}.  Among several more
important results in his work, the following caught our attention: for
each interaction label $v$ of a quantum field theory, the Connes--Kreimer
Hopf algebra of Feynman graphs contains a formal series $Y_v$ satisfying the
multi-variate `\fdb' formula
\begin{equation}\label{vS:Y}
\Delta(Y_v) = 
\sum_{n_1\cdots n_k} Y_v Y_{v_1}^{n_1}\cdots Y_{v_k}^{n_k} 
\otimes p_{n_1\cdots n_k} (Y_v) ,
\end{equation}
where 
$p_{n_1\cdots n_k}$ is the projection onto graphs containing $n_i$
vertices of type $v_i$.  
The series $Y_v$ is the renormalised (combinatorial) 1PI Green function
$$
Y_v = \frac{G_v}{\prod_{e\in v} \sqrt{G_e}} ,
$$
where
$$
G_v = 1 + \sum_{\operatorname{res}\Gamma=v} \frac{\Gamma}{\norm{\Aut\Gamma}}
$$
is the bare Green function of all connected 1PI graphs $\Gamma$ with residue $v$, 
the product is over the lines of the one-vertex graph $v$, and where
the denominators
$$
G_e = 1 -\sum_{\operatorname{res}\Gamma=e} \frac{\Gamma}{\norm{\Aut\Gamma}}
$$
constitute a renormalisation factor, cf.~the Dyson formula
(see \cite[Ch.~8]{Itzykson-Zuber:QFT}) or \cite[Ch.~7]{Kaku}).
% We refer to \ref{comparison-vS} below for further explanation.
% 
% 
% \boxnote{
% THE SUM STARTS WITH 1!
% He has collapsed all the interaction labels to 1.
% That would correspond to collapsing all edge colours in the tree case.
% We thought this could not give a reasonable Green function?!  Need to check 
% this out better\ldots
% }
Van Suijlekom's proof of the formula is a matter of expanding everything,
keeping track of several different combinatorial factors associated to graphs,
and comparing them with the help of the orbit-stabiliser theorem.  (The formula
is Proposition~12 of \cite{vanSuijlekom:0807}, but the bulk of the proof is
contained in various lemmas in \cite{vanSuijlekom:0610} where the 
combinatorial factors involved are computed.)

Interest in Green functions in
Hopf algebras of graphs is due in particular to the fact that,
unlike the individual graphs, the Green
functions actually have a physical interpretation.
The \fdb Hopf algebra plays an important role in Hopf algebra approach
to renormalisation,
and many different relationships between it and the Hopf algebras of graphs or 
trees have been uncovered.
One reason for the importance of the \fdb  Hopf algebra is the general idea,
 expressed for example by Delamotte~\cite{Delamotte}, that in the end renormalisation
should be a matter of reparametrisation, i.e.~substitution of power series.

Already Connes and Kreimer~\cite{Connes-Kreimer:MR1810779}
constructed a Hopf algebra homomorphism from the \fdb Hopf algebra (or 
rather the Connes--Moscovici Hopf algebra) to the Hopf algebra of Feynman 
graphs in the case of $\phi^3$ in six space-time dimensions.
Bellon and 
Schaposnik~\cite{Bellon-Schaposnik:0801.0727} were perhaps the first to
explicitly write down the \fdb formula, in a form
$$
\Delta(a) = \sum_n a^n \otimes a_n ,
$$
very pertinent to the formula we establish in the present paper.
Recently the \fdb formula has been exploited by
Ebrahimi-Fard and Patras~\cite{Ebr_F-Patras} in the development of
exponential renormalisation.  Their paper contains also valuable
information on the relationship with the Dyson formula.

It seems unlikely that a formula like this can exist for the Green function in
the Hopf algebra of {\em trees}
--- indeed, the symmetry factors
of the trees involved are not related to the combinatorics of grafting in the
same way as symmetry factors of graphs are related to insertion of graphs
(except in very special cases, such as considering only
iterated one-loop self-energies in massless Yukawa theory in four dimensions,
an example considered by many authors, e.g.~\cite{Connes-Kreimer:MR1915297}, 
\cite{Broadhurst-Kreimer:hep-th/0012146},
\cite{Kreimer-Yeats:MR2255485}).

In the present paper we work with {\em operadic trees} instead of the
combinatorial trees of the usual (Butcher)--Connes--Kreimer Hopf algebra --- this is an
essential point: operadic trees are more closely related to Feynman graphs, and
have meaningful symmetry factors in this respect,
cf.~\cite{Kock:graphs-and-trees} (see also~\ref{trees-of-graphs} 
below).

Our main theorem (\ref{main-thm-alg}) at the algebraic level
establishes the \fdb formula
\begin{equation}\label{main-formula-intro}
\Delta(G) = \sum_n G^n \otimes p_n(G) 
\end{equation}
for the Green function $G=\sum_T T/\norm{\Aut(T)}$
in the bialgebra of $P$-trees, for any polynomial endofunctor $P$.  

The proof we give is very conceptual:
 the equation
appears as an equivalence of groupoids, and all the symmetry factors
are hidden and take care of themselves.
A few remarks may be in order here to explain how this works.

  A basic construction in combinatorics is to split a set into a disjoint union
  of parts: given a map of sets $E\to B$, the `total space' $E$ is the sum of the
  fibres:
  $$
  E = \sum_{b\in B} E_b .
  $$
  The same formula holds for groupoids, with the appropriate homotopy notions:
  given a map of groupoids
  $E\to B$, there is a natural equivalence
  $$
  E \simeq \grint{b}{B}{E_b} .
  $$
  The integral sign denotes the {\em homotopy sum} of the family (see \ref{hosum})
  (and the fibres are homotopy fibres).
  Up to non-canonical
  equivalence it can be computed as
  $$
  \simeq \sum_{b\in\pi_0B}E_b/\Aut b,
  $$
  revealing the symmetry factors, but our point is that homotopy sums interact
  very nicely with homotopy pullbacks, making the formalism look exactly as if
  we were dealing just with sets, and it is never necessary to mention the symmetry
  factors explicitly.

Our main theorem (\ref{main-thm-grpd}) at the groupoid level states
the following equivalence of groupoids over $\forest\times\tr$:
\begin{equation}\label{main-equiv-intro}
\grint{T}{\tr}{\cut(T)} \simeq \grint{N}{{\fam I}}{\forest_{\leaves{N}} 
\times {}_{\leaves{N}}\tr} ,
\end{equation}
which is essentially a double-counting formula.
Here 
%an integral sign denotes a certain homotopy colimit (corresponding to a sum, with symmetry factors) of family of groupoids,  
$\cut(T)$ is the discrete groupoid of cuts of a tree, 
$\leaves N$ is an ($I$-coloured) set,
$\forest_{\leaves{N}}$ is the groupoid of forests with root profile $\leaves N$, and
${}_{\leaves{N}}\tr$  is the groupoid of trees with leaf profile $\leaves N$.
More precisely, if $\forest$ and $\tr$ are the groupoids of $P$-forests and $P$-trees, then $\forest_{\leaves{N}}$ and ${}_{\leaves{N}}\tr$ are the {\em homotopy fibres} over $\leaves{N}$ of the root and leaf functors respectively.  
% 
% Our second proof lifts the same
% arguments to  an equivalence in the bimonoidal category of
% polynomial functors in tree-many variables, cf.~\cite{{Kock:1109.5785}.  
The algebraic 
\fdb formula \eqref{main-formula-intro}
is obtained just by taking homotopy cardinality  (relative 
to $\forest\times\tr$)
on both sides of
the equivalence \eqref{main-equiv-intro}.

In order to arrive at a level of abstraction where the arguments become pleasant
and the essential features are in focus, we have moved away quite a bit from the
starting point mentioned above, and at the moment we have not quite succeeded in
deriving van Suijlekom's formula from ours (or conversely).  Depending on the
choice of polynomial endofunctor $P$, our formula specialises to various
formulae of independent interest, such as formulae for planar trees or binary
trees.  Our motivating example of polynomial endofunctor $P$, explained at the
end of the paper, is defined in terms of interaction labels and 1PI graphs for
any quantum field theory.  Via work in progress by the second-named
author~\cite{Kock:graphs-and-trees} establishing a bialgebra homomorphism to
this bialgebra of $P$-trees from the bialgebra of graphs, we hope in subsequent
work to be able to derive van Suijlekom's formula from the \fdb formula of the
present paper.

\begin{blanko*}{Outline of the paper.}

%1
  Section~\ref{sec:clas} and~\ref{sec:cktrees} are mostly motivational.
  We begin in Section~\ref{sec:clas} by revisiting the classical \fdb
  Hopf algebra, gradually recasting it in more categorical language,
  starting with composition of formal power series, then the incidence
  algebra viewpoint (cf.~\cite{DoubiletMR0335568}), then finally the
  category of surjections (cf.~\cite{JoyalMR633783}).  We work with
  the non-reduced {\em bi}algebra rather than with the reduced {\em
    Hopf} algebra. This is an important point.
%2 
  In Section~\ref{sec:cktrees} we briefly revisit the (Butcher)--Connes--Kreimer Hopf
  algebra of trees, introduce an operadic version of it that we need, and state
  one version of the main theorem for the bialgebra of operadic trees and the
  corresponding Green function.

  %3
  The theory of groupoids is at the same time our main 
  technical tool and the most important conceptual ingredient in our approach.
  Section~\ref{sec:groupoids} recalls a few notions, fixing terminology and 
  notation for homotopy pullbacks, fibres,  quotients and
  sum, in the hope of rendering the paper accessible to
  readers without a substantial background in category theory.
%4
  In Section~\ref{sec:trees} we set up the formalism of operadic trees
  and forests, in terms of polynomial endofunctors, following
  \cite{Kock:0807}.  This formalism is needed in particular to be able
  to talk about decorated trees --- $P$-trees for a polynomial
  endofunctor $P$ --- at the level of generality needed to cover the
  examples envisaged.  

%5
In Section~\ref{sec:fdbequiv} we establish our main result,
the equivalence of groupoids over $\forest\times\tr$:
$$
\grint{T}{\tr}{\cut(T)} \simeq \grint{\leaves{N}}{\fam I}{
%\Grpd/I(\leaves{n},\tr)
  \forest_{N}
\times {}_{\leaves{N}}\tr} 
$$
already mentioned.
Most of the arguments are formal consequences of
general properties of groupoids; the only thing
we need to prove by hand is the equivalence
$$
\treescuts\simeq\forest\times_{\fam I}\tr 
$$
between trees with a cut
and pairs consisting of a forest and a tree such
that the roots of the forest `coincide' with
the leaves of the tree (Lemma~\ref{tree:key}).  In a precise sense,
this is the essence of the 
Hopf algebra of trees.

%6
  Section~\ref{Sec:card} reviews and extends appropriate notions of groupoid
  cardinality, following Baez--Dolan~\cite{Baez-Dolan:finset-feynman} and
  Baez--Hoffnung--Walker~\cite{Baez-Hoffnung-Walker:0908.4305}.  In particular,
  we establish the basic properties of
  relative cardinality with respect to a morphism of groupoids.
  In Section~\ref{fdbalg} we finally prove the
  \fdb formula in the bialgebra of trees by taking cardinality of the
  groupoid equivalence of Section~\ref{sec:fdbequiv}.
%We begin by introducing a general enough notion of Green function as a relative cardinality of a map. Particular cases give rise to several Green functions.
 %
% Next we state and prove the Faà di Bruno formula for the bialgebra of trees. using the groupoid constructions previously introduced.
%and taking cardinalities

Examples of polynomial endofunctors giving rise to several kinds of trees 
are given in Section~\ref{sec:ex}. In particular we relate our \fdb formulae 
with
the classical one.  In our final example we describe a
polynomial endo\-functor $P$ defined in terms of Feynman graphs, which
points towards transferring our results to bialgebras of graphs.
\end{blanko*}

\begin{blanko}{Acknowledgments.}
We are indebted to Kurusch Ebrahimi-Fard for
many illuminating discussions on quantum field theory,
and to the anonymous referee whose comments led to some expository improvements. 
\end{blanko}

%%%%%%%%%%%%%%%%%%%%%%%%%%%%%%%%%%%%%%%%%%%%%%%%%%
\section{The \fdb formula revisited}
%%%%%%%%%%%%%%%%%%%%%%%%%%%%%%%%%%%%%%%%%%%%%%%%%%

\label{sec:clas}

In this section we briefly review the classical \fdb
bialgebra, first in terms of composition of power series,
then in terms of partitions, and finally in terms of the
groupoid of surjections.

\begin{blanko}{Power series and the classical Fa{\`a} di Bruno formula.}
Consider formal power series in one variable without
  constant term and with linear term equal to $z$:
\begin{align*}
f(z)&=\sum_{n=0}^\infty \frac{A_n(f)\,z^n}{n!} \qquad A_0=0,\ A_1=1 .
\end{align*}
These form a group under substitution of power series,
sometimes denoted
$\operatorname{Diff}(\C,0)$, as the series can be regarded as germs of
smooth functions tangent to the identity at $0$.
  The classical \fdb Hopf algebra
$\fdbhopf$ is the polynomial algebra on the symbols
$$
a_n := A_n/n!, \quad n\geq 2,
$$
viewed as linear forms on $\operatorname{Diff}(\C,0)$, 
$$
\langle a_n,f\rangle=a_n(f) = A_n(f)/n!,\qquad a_n\in \C[[z]]^*.
$$
The comultiplication is defined by
\begin{align*}
\langle\Delta (a_n) ,f\otimes g\rangle&=\langle a_n,g\circ f\rangle ,
\end{align*}
and the counit by $ \varepsilon(a_n)=\langle a_n,{\mathbf1}\rangle $.
An explicit formula for $\Delta$ can be obtained by expanding
\begin{align}\label{formalcomp}
(g\circ f)(z)
&=\sum_{n=1}^\infty a_n(g)\biggl(
\sum_{m=1}^\infty a_m(f)\,z^m
\biggr)^{\!\!n},
\end{align}
and involves the Bell polynomials.  So far $\fdbhopf$ is a bialgebra; it
acquires an antipode by general principles by observing that it is a connected
graded bialgebra: the grading is given by
$$
\deg(a_k) = k-1 .
$$
We refer to 
Figueroa and Gracia-Bond\'{\i}a~\cite{Figueroa-GraciaBondia:0408145} for
details on this classical object and its relevance in quantum field theory.

The formula for $\Delta$ can be packaged into a single equation, by
considering the formal series
$$
A = 1+ \sum_{k\geq 2} \frac{A_k}{k!} = 1+ \sum_{k\geq 2} a_k  \quad \in \C[[a_2,a_3,\ldots]] .
$$
The comultiplication extends to series, and now takes the following form:
$$
\Delta(A) = A \otimes 1 + \sum_{k\geq 2} A^k \otimes a_k .
$$
The values of $\Delta$ on the individual generators $a_k$ can be extracted from 
this formula.
\end{blanko}

\begin{blanko}{The (non-reduced) \fdb bialgebra.}
    \label{fdb}
  For our purposes it is important to give up the condition $a_1=1$.  In this
  case, substitution of power series does not form a group but only a monoid,
  and the algebra is just a bialgebra rather than a
  Hopf algebra.  We denote it by $\fdbsymb = \C[a_1,a_2,a_3,\ldots]$.  The
  definition of the comultiplication is still the same, and again it can be
  encoded in a single equation, involving now the formal series
  $$
  A = \sum_{k\geq 1} \frac{A_k}{k!}= \sum_{k\geq 1} a_k  \quad \in \C[[a_1,a_2,a_3,\ldots]].
  $$
  The resulting form of the \fdb formula is the Leitmotiv of the present work:
  \begin{prop}[Classical \fdb identity]\label{fdb-clas}The formal series $A$ satisfies
  $$
  \Delta(A) = \sum_{k\geq 1} A^k \otimes a_k  .
  $$    
  \end{prop}
% We will enrich this result to an equivalence of groupoids in the next section,
% and give a proof that is easy to translate to the bialgebra of trees.

%  \boxnote{With the obvious convention $a_0=0$, we may as well allow the
%  sum to start at $k=0$, which is what generalises better, cf.~our main theorem.}

%\normalsize

  We stress that the bialgebra $\fdbsymb$ (with grading $\deg(a_k)=k-1$) is not
  connected: $\fdbsymb_0$ is spanned by the powers of $a_1$, all of which are
  group-like.  One can obtain the classical Hopf algebra $\fdbhopf$ by imposing
  the relation $a_1=1$, which is easily seen to generate a bi-ideal.  
\end{blanko}

% There is an explicit formula for the comultiplication.
% It involves the Bell-polynomials, which in turn are the
% generating functions for the numbers of partitions of given shapes.
% All this combinatorics can be condensed into the following formula.
% In the reduced case, let $A = 1+ \sum_{k\geq 2} a_k$.  Then
% $$
% \Delta(A) = A \otimes 1 + \sum_{k\geq 2} A^k \otimes a_k
% $$
% The values of $\Delta$ on the individual generators $a_k$ can be extracted from 
% this formula.

\begin{blanko}{Note on grading convention.}
Since $\deg(a_k)=k-1$, it is common in the literature to employ a different 
indexing, shifting the index so that it agrees with the degree.  
With the shifted index convention, the \fdb formula then reads
$$
\Delta(A) = \sum_{n\geq 1} A^{n+1} \otimes a_n .
$$
This is the convention used by 
van Suijlekom and many others, and explains the extra factor $Y_v$ in
the formula~\eqref{vS:Y} quoted above.
Beware that this convention means that certain indices are allowed to
start at $-1$ and when it is said that $p_n(G)$ is the part of the Green function
corresponding to graphs with $n$ vertices, it actually means $n+1$ vertices.

While the shifted indexing convention
can have its advantages, it is important for us to keep the indexing as above, so that the
exponent in $A^k$ matches the index in $a_k$.  As we pass to more
involved \fdb formulae, this will always express a type match: the outputs of
one operation (the exponent) matching the input of the following (the subscript).
% 
% 
% 
% The convention used in the present paper is dictated by the operadic viewpoint:
% the summands represent cuts (of surjections or of trees), and
% the number of outputs of the left-hand factor (the exponent)
% should match the number of inputs
% of the right-hand factor (which is the subscript).
% 
% On the other hand, we don't know if there is a deeper significance hidden in this
% index shift.  
\end{blanko}

\begin{blanko}{Fa{\`a} di Bruno Hopf algebra in terms of partitions.}
  The coefficients --- the Bell polynomials which we did not make explicit ---
  count partitions.  In fact, it is classical 
  (Doubilet~\cite{DoubiletMR0335568}, 1975) that the Hopf algebra 
  $\fdbhopf$
  can be realised as the reduced {\em incidence bialgebra} of the family of posets
  given by partitions of finite sets.
%   This was first discovered by
%   Doubilet~\cite{DoubiletMR0335568}; we refer to Schmitt~\cite{Schmitt:1994} for
%   the details.
%  We give only a brief outline here.
  
  The partitions of a finite set $S$ form a lattice, in which $\sigma\leq 
  \tau$ when $\sigma$ is a refinement of $\tau$. 
  Consider the family of all intervals $[\sigma,\tau] := \{\rho \mid \sigma\leq \rho 
  \leq \tau \}$ in partition lattices of finite sets, and 
  declare two intervals equivalent if they are isomorphic as abstract posets.
  This is an order-compatible equivalence relation, meaning that the 
  comultiplication formula
  $$
  \Delta([\sigma,\tau]) = \sum_{\rho\in [\sigma,\tau]} [\sigma,\rho]\otimes 
  [\rho,\tau]
  $$
  is well-defined on equivalence classes.  Disjoint union of finite sets defines
  furthermore a multiplication on these equivalence classes.  If $a_k$ is the interval given by the partition lattice of a set with $k\geq2$ elements, then any
  interval is equivalent to a finite product of such $a_k$ and this product
  expression is unique up to isomorphism of the sets involved.
  
  The reduced incidence coalgebra on the vector space spanned by all equivalence classes (that is,  
  the polynomial ring on the classes $a_k$, $k\geq2$)
  is naturally isomorphic to the \fdb Hopf algebra 
  $\fdbhopf$.
  
  In order to get the `nonreduced' bialgebra $\fdbsymb$, one has to 
  consider a finer equivalence relation: define an 
  interval $[\sigma,\tau]$ to have {\em type} $1^{\lambda_{1}} 2^{\lambda_{2}} \cdots$ if 
  $\lambda_k$ is the number of blocks of $\tau$ that consist of exactly $k$ blocks 
  of $\sigma$,
  and declare two intervals equivalent if they have the 
  same type.  Every interval is isomorphic as a poset to a 
  type-equivalent product of (possibly trivial)
  maximal intervals, yielding a `nonreduced' incidence algebra 
  isomorphic to $\fdbsymb$.
  The 
%subtleties
technicalities involved here
%in the intermediate step of considering
%  intervals and their equivalences 
can be avoided by considering
  surjections instead of partitions.
\end{blanko}

\newcommand{\shortsur}{\nolinebreak[4]\!\onto\nolinebreak[4]\!}
\begin{blanko}{\fdb in terms of surjections.}
%
%  A surjection $E\shortsur B$ clearly induces a partition of the set $E$, and
%  conversely, a partition of $E$ induces a surjection to the set
%  of blocks.  This correspondence provides a groupoid equivalence between
%  the groupoid of
%%%%%%%%%%%%%%% (nonempty) %%  unit is basis elt corresponding to interval of partitions of \empty  
%sets-with-a-partition and their isomorphisms, and the
%  groupoid $\surj$ of 
%%%%%%%%%%%%%%% (nonempty) 
% surjections.  The arrows in the groupoid 
%  $\surj$ are pairs of isomorphisms forming a diagram
%  $$
%  \xymatrix{
%  E\ar@{->>}[d] \ar[r]^\cong  &  E'\ar@{->>}[d]  \\
%  B\ar[r]^\cong   & B' .
%  }
%  $$
Considering surjections $E\shortsur B$ of nonempty finite sets, 
one can get the bialgebra $\fdbsymb$ directly.
% (without the intermediate step of intervals).  
%%%  cheating?   Fact(E->>B)s are intervals!
As a vector space it has as basis
the isomorphism classes of surjections. 
The multiplicative structure is
%comes from the monoidal structure on $\surj$
given by disjoint union, and since any surjection is the disjoint union of {\em connected} 
surjections $a_k=(\{1,\dots,k\}\shortsur \{1\})$,
we have $\fdbsymb\cong\C[a_1,a_2,\dots]$.

The comultiplicative structure is given by
$$
\Delta(E\shortsur B) = \sum_{\substack{[E \onto S\onto B]\;\;\in\\[0.6ex]
\pi_0\operatorname{Fact}(E\onto B)}}
\!\!\!\!
(E\shortsur S) \otimes 
(S\shortsur B).
$$
%Here the sum is over isomorphism classes of factorisations $E \onto S \onto B$.
%In detail, 
%consider 
Here the sum is over the components of the {\em factorisation groupoid} 
$\operatorname{Fact}(E\shortsur B)$, which has as objects the factorisations of 
$E\shortsur B$ into two surjections $E\shortsur S\shortsur B$,
and  as morphisms the diagrams:
%bijections $S\simeq S'$ making the two triangles commute:
$$
\xymatrix@R1ex{
&S\ar[dd]^{\simeq}\ar@{->>}[rd]& \\
E\ar@{->>}[ru]\ar@{->>}[rd] && B \,.\\
& S'\ar@{->>}[ru]
}$$
%Then the above sum is over $\pi_0(\operatorname{Fact}(E\shortsur B))$,
%the set of connected components of the factorisation groupoid.
%It is clear that the ??number?? of factorisations does not depend on the 
%representative of the isoclass of surjections.
The relation with partitions is clear: 
a surjection $E\shortsur B$  induces a partition of the set $E$, and
a partition of $E$ induces a surjection to the set of parts.  This correspondence provides an equivalence between the groupoid $\surj$ of surjections and that of sets-with-a-partition.

%This gives
%$$
%\operatorname{Fact}(E\shortsur B)
% = \operatorname{Fact}\big(\sum_{b\in B} (E_b 
%\shortsur \{b\})\big)
%= \prod_{b\in B} \operatorname{Fact}(E_b\shortsur \{b\}) .
%$$
%Also $\operatorname{Fact}(E\shortsur B)$ is isomorphic to the product of $\operatorname{Fact}(E_b\shortsur 1)$.
%It follows that $\fdbsymb$ indeed coincides with the classical \fdb bialgebra.

To obtain the \fdb Hopf algebra $\fdbhopf$ 
 we identify surjections with equivalent factorisation groupoids, rather than just isomorphic surjections.
Thus invertible surjections are all equivalent, as they have trivial factorisation groupoids.
This relation is clearly generated by the equation
$(1\shortsur 1)=(\emptyset\shortsur \emptyset)$, that is, $a_1=1$.

The construction of the \fdb bialgebra in terms of the groupoid of
surjections seems to be due to 
Joyal~\cite{JoyalMR633783}.
It is in the spirit of incidence algebras of M\"obius 
categories introduced by Leroux~\cite{Leroux:1975}, and studied recently by
Lawvere and Menni~\cite{LawvereMenniMR2720184}.  However, the category of surjections
is not a M\"obius category, since it contains non-trivial isomorphisms.
%The theory of {\em homotopy M\"obius categories} that we develop in
In our forthcoming paper~\cite{GKT2} we extend the classical theory of incidence algebras and M\"obius categories by allowing groupoid coefficients in order to cover  the category of
surjections, and also the category of trees in 
Section~\ref{sec:trees} below.
\end{blanko}

\section{The bialgebra of trees, and the %first version of  
Main Theorem}
%%%%%%%%%%%%%%%%%%%%%%%%%%%%%%%%%%%%%%%%%%%%%%%%%%
\label{sec:cktrees}

\begin{blanko}{The bialgebra of rooted trees of Connes and 
  Kreimer~\cite{Kreimer:9707029}}, which in essence was studied 
  already by Butcher~\cite{Butcher:1972} in the early 70s,
  is the free %$\C$-
  algebra $\CK$ on the set of isomorphism classes of
  {\em combinatorial trees} (defined for example as finite connected graphs without
  loops or cycles, and with a designated root vertex). 
  The comultiplication is
  given on generators by
  \begin{eqnarray*}
  \Delta:  \CK & \longrightarrow & \CK \otimes \CK  \\
    T & \longmapsto & \sum_c P_c \otimes S_c ,
  \end{eqnarray*}
  where the sum is over all admissible cuts of $T$; the left-hand factor $P_c$
  is the forest (interpreted as a monomial) found above the cut, and $S_c$ is
  the subtree found below the cut (or the empty forest, in case the cut is below
  the root).  Admissible cut means: either a subtree containing the root, or the empty set.
  $\CK$ is a connected bialgebra: the grading is by the number of nodes, and
  $\CK_0$ is spanned by the unit.  Therefore, by general principles (see for 
  example \cite{Figueroa-GraciaBondia:0408145}), it acquires
  an antipode and becomes a Hopf algebra.
\end{blanko}

\begin{blanko}{Operadic trees.}
  For the present purposes it is crucial to work with {\em operadic trees}
  instead of combinatorial trees; this amounts to allowing loose
  ends (leaves). A formal definition is given in \ref{polytree-def}.  For the
  moment, the following drawings should suffice to exemplify operadic trees ---
  as usual the planar aspect inherent in a drawing should be disregarded:
  \begin{center}\begin{texdraw}
  \linewd 0.5 \footnotesize
  \move (-50 0)
  \bsegment
    \move (0 0) \lvec (0 30)
  \esegment
  
  \move (0 0)
  \bsegment
    \move (0 0) \lvec (0 18) \onedot
  \esegment
  
  \move (50 0)
  \bsegment
    \move (0 0) \lvec (0 36)
    \move (0 18) \onedot
  \esegment
  
  \move (105 0)
  \bsegment
    \move (0 0) \lvec (0 15) \onedot
    \lvec (-5 33) \onedot
    \move (0 15) \lvec (-12 28) \onedot
    \move (0 15) \lvec (4 43)
    \move (0 15) \lvec (12 40)
  \esegment
  
  \move ( 170 0)
  \bsegment
    \move (0 0) \lvec (0 18) \onedot
    \lvec (-6 32) \onedot
    \lvec (-12 57)
    \move (0 18) \lvec (4 40) \onedot
    \lvec (20 50) \onedot
    \lvec (15 65)
    \move (20 50) \lvec (25 65)
    \move (4 40) \lvec (9 54) \onedot
    \move (4 40) \lvec (-4 61)
  \esegment
  \end{texdraw}\end{center}
  Note that certain edges (the {\em leaves}) do not start in a node, and that one
  edge (the obligatory {\em root edge}) does not end in a node.  A node without incoming
  edges is not the same thing as a leaf; it is a nullary operation (i.e.~a
  constant), in the sense of operads.  In operad theory, the nodes represent
  operations, and trees are formal combinations of operations.  The small
  incoming edges drawn at every node serve to keep track of the arities of the
  operations.  Furthermore, for coloured operads, the operations have type
  constraints on their inputs and output, encoded as attributes of the edges.

  The trees appearing in BPHZ renormalisation are naturally
  operadic, as the nodes and edges come equipped with decorations by
  the graphs encoded.  This is briefly  explained in 
  Example~\ref{trees-of-graphs}, following~\cite{Kock:graphs-and-trees}.
\end{blanko}

\begin{blanko}{The bialgebra of operadic trees (cf.~\cite{Kock:1109.5785}).}
  \label{treebialg}
  A {\em cut} of an operadic tree is defined to be a subtree containing the root
  --- note that the arrows in the category of operadic trees
  are arity preserving (\ref{sub}), meaning that if
  a node is in the subtree, then so are all the incident edges of that node.
  
  If $c:S\subset T$ is a subtree containing the root, then each leaf $e$ of
  $S$ determines an ideal subtree of $T$ (\ref{sub}), namely consisting of $e$ (which
  becomes the new root) and all the edges and nodes above it.  This is still
  true when $e$ is a leaf of $T$: in this case, the ideal tree is the
  trivial tree consisting solely of $e$.  Figuratively, this means
  that for operadic trees cuts are not allowed to go above the leaves,
  and that cutting an edge does not remove it, but really cuts it(!).
  Note also that the root
  edge is a subtree;  the ideal tree of the root edge is of course the tree
  itself.  This is the analogue of the cut-below-the-root in the combinatorial
  case.  For a cut $c:S\subset T$, define $P_c$ to be the forest consisting
  of all the ideal trees generated by the leaves of $S$.
  
  Let $\treebialg$ be the free algebra (that is, the polynomial ring) on the set of isomorphism 
  classes of operadic trees, with comultiplication defined on the generators
  by
  \begin{eqnarray*}
  \Delta:  \treebialg & \longrightarrow & \ \ \treebialg \otimes \treebialg  \\
    T & \longmapsto & \sum_{c:S\subset T}\!\! P_c \otimes S.
  \end{eqnarray*}
  As for combinatorial trees, $\treebialg$ becomes a graded bialgebra, but it is
  not connected since $\treebialg_0$ is spanned by all powers of the trivial tree
  $\inlineDotlessTree$.  These are grouplike, so one could obtain a connected bialgebra by
  imposing the equation $\inlineDotlessTree=1$.
\end{blanko}

\begin{blanko}{The Green function.}\label{tree:green}
  In the completion of $\treebialg$ (that is, the power series ring),
  the series
  $$
  G := \sum_{T} \delta_T/\norm{\Aut(T)}
  $$
  is called the {\em Green function}, in analogy with the (combinatorial)
  Green function of Feynman
  graphs.  The sum is over all isomorphism classes of (operadic) trees,
   and there is a formal symbol $\delta_T$ for each isomorphism class of trees.
%  We shall soon consider decorated trees, in which case there is one
%  Green function for each possible decoration of the root edge,
\end{blanko}

The following \fdb formula for the Green function in the
  bialgebra of (operadic) trees is a special case of our main 
  theorem~(\ref{main-thm-alg}).
\begin{thm}\label{tree:delta}
  
%   $$
%   \Delta(G)=\sum_{S} G^{LS} \otimes S/\norm{\Aut(S)} .
%   $$
%   The sum is over all isomorphism classes of trees, and $LS$ denotes the set of
%   leaves of a tree $S$.  There is also a less refined formula, but which 
%   resembles more the original \fdb formula:
  Write $G= \sum_{n\in\N} 
  g_n$, where $g_n$ is the summand in the Green function corresponding to
  trees with $n$ leaves.  Then
  $$
  \Delta(G)=\sum_{n\in \N} G^n \otimes g_n .
  $$
\end{thm}

The more general formula we prove is valid for $P$-trees for any polynomial 
endofunctor $P$.  In addition to the naked trees considered so far,
this covers many  examples  such as
planar trees, binary trees, cyclic trees (Example~\ref{example:cyclic}),
as well as the trees decorated by
connected 1PI graphs of a quantum field theory (Example~\ref{trees-of-graphs}).
% The latter allows to transfer
% the \fdb formula to a bialgebra of graphs.

% \begin{blanko}{From operadic trees to combinatorial trees.}
%   \label{operadic-v-combinatorial}
  It is essential that we use operadic trees.  There seems to be no reasonable
  Green function for combinatorial trees, since their symmetry factors are not
  related to the combinatorics of grafting.

% FLAGGED FOR DELETION THE REST OF THIS
%   The {\em core} of a non-trivial operadic tree $T$ is the combinatorial tree
%   $T\upperdot$ obtained by pruning off all leaves as well as the root edge.
%   Taking core is functorial in root-preserving inclusions, and induces a 
%   bialgebra homomorphism from the bialgebra of operadic trees to the Hopf 
%   algebra of combinatorial trees \`a la Connes--Kreimer.
  
%   For surjections, the Green function could survive the reduction since the only
%   killed variable was $a_1$, which could be reconstructed by a homogeneity
%   argument.  For trees, the reduction (taking core of operadic trees) destroys
%   too much information.

%   The Green function cannot be defined directly in terms of combinatorial trees
%   (or it would not have much relevance).  However, it is possible of course to
%   define it in terms of operadic trees, and then perform the reduction: Since
%   taking core is a bialgebra homomorphism, this yields a series in $\CK$.  The
%   Fa\`a di Bruno formula also carries over but does not seem to have any
%   immediate interpretation of interest.  ????

% \end{blanko}

\bigskip

We now first need to review some standard groupoid theory, then introduce
more formally the trees and $P$-trees we treat, before coming to the proofs.

%%%%%%%%%%%%%%%%%%%%%%%%%%%%%%%%%%%%%%%%%%%%%%%%%%
\section{Groupoids}
%%%%%%%%%%%%%%%%%%%%%%%%%%%%%%%%%%%%%%%%%%%%%%%%%%
\label{sec:groupoids}
We recall some standard facts about groupoids, emphasising
the use of the correct {\em homotopy} notions
of the basic constructions such as pullback, fibre, quotient and sum.
Although each of these notions can be traced a long way back (e.g.~\cite{SGA1}, 
\cite{Hakim}, \cite{Brown-Heath-Kamps:Toledo}),
the consistent use of them in applications to combinatorics seems to be new.
It is the systematic use of homotopy sums that makes all the symmetry factors
`disappear'.  
%(A different, related, approach was taken by
%Baez--Hoffnung--Walker~\cite{Baez-Hoffnung-Walker:0908.4305}, bulding 
%instead the automorphism groups into the fibres via the
%notion of full fibre, and combining with sums over $\pi_0$ of
%groupoids.  But since the notion of full fibre is not stable
%under pullback, the auts tend to pop out again as correction
%factors, as can be observed in their work~\cite{Baez-Hoffnung-Walker:0908.4305}%.)

% 
% These were inspired by analogous classical notions in topology (such as the double mapping path space) and can be
% from the beautiful simplicial machinery developed by
% Joyal~\cite{Joyal:qCat+Kan} to generalise the theory of categories to
% quasi-categories (called $\infty$-categories by Lurie~\cite{Lurie:HTT}).

\begin{blanko}{Basics.}
  A {\em groupoid} is a category in which every arrow is invertible.  A {\em
morphism} of groupoids is a functor, and we shall also need their natural
isomorphisms.  
%While category theory language is the main technical tool to
%deal with groupoids, the intuition is rather that groupoids are `fat sets with
%symmetries': instead of having just a few isolated points (elements in a set) we
%now have large chunks of points which are equivalent, with specific arrows
%linking them up.  More than one arrow can exist between two given objects, and
%indeed a single object can have more than one arrow to itself --- these are its
%symmetries.
%
%%A group is a groupoid with only one object, and 
%%A set may be considered a groupoid with only identity arrows. 
%This defines a functor
%$$
%D: \Set \to \Grpd.
%$$
%Conversely, a groupoid $X$ gives rise to a set by taking its set of connected
%components, i.e.~the set of isomorphism classes in $X$, denoted $\pi_0(X)$;
%this defines a functor in the other direction (the left adjoint of $D$)
%$$
%\pi_0 : \Grpd \to \Set.
%$$
The set of isomorphisms classes, or components, of a groupoid $X$ is denoted $\pi_0X$.  Many sets arising in combinatorics and physics, such as `the set of all trees', are actually sets of isomorphism classes of  a groupoid.
For each object $x$ the {\em vertex
group}, denoted $\pi_1(x)$ or $\Aut(x)$, consists of all the arrows from
$x$ to itself.
% The homotopy notations $\pi_0$ and $\pi_1$ from topology are not a whim: 
% groupoids are in a precise sense a model for certain topological spaces, namely
% the homotopy $1$-types.  To a topological space one associates the 
% fundamental groupoid, whose objects are the points of the space and whose arrows
% are the (homotopy classes of) paths between points.  Conversely, from a groupoid
% $X$ one can build a CW complex, the classifying space $BX$, whose fundamental groupoid is $X$ and which has
% %taking the objects as $0$-cells, the arrows as $1$-cells, and filling all
% %commutative triangles in the groupoid.  The filling means that only spaces with
% vanishing higher homotopy groups ($\pi_k=0$ for $k\geq 2$):
% % arise in this way
% these spaces are called homotopy $1$-types.  
%  
The notation $\pi_0$, $\pi_1$ is from topology. 
The homotopy viewpoint of groupoids
is an important aspect, as all the good notions to deal with them are
homotopy notions (e.g.~homotopy pullback, homotopy fibres, homotopy quotients, 
etc.), as we proceed to recall.
\end{blanko}

\begin{blanko}{Equivalence.}
An {\em equivalence} of groupoids is just an equivalence of categories, i.e.~a
  functor admitting a pseudo-inverse.  Pseudo-inverse means that the two
  composites are not necessarily exactly the identity functors, but are only
  required to be isomorphic to the identity functors. 
  A morphism of groupoids is an equivalence if and only if it induces a 
  bijection
  on $\pi_0$, and an isomorphism at the level of $\pi_1$.
% This is the analogue of a
%  homotopy equivalence in topology.  Like in category theory, equivalences of
%  groupoids can also be characterised as functors which are fully faithful and
%  essentially surjective.  
%
%Just as sets are often only interesting up to bijection, the appropriate notion
%of sameness for groupoids is equivalence.  Equivalent groupoids have the same
%properties, for example the same cardinality in the sense
%of~Section~\ref{Sec:card} below.
A groupoid $X$ is called {\em discrete} if it is equivalent to a set
(that is, its vertex groups are all trivial), and
%is equivalent to a set considered as 
%a groupoid; this set can then be taken to be $\pi_0(X)$.
%Another way of saying the same is that all vertex groups are
%trivial: $\pi_1(x)=1$ for all objects $x\in X$, so all the information is 
%stored in $\pi_0$.
% Indeed the condition can be stated by saying that a 
% groupoid is discrete if it is equivalent to its $\pi_0$.
%(There is a potential
%risk of confusion with the word `discrete': in settings where one considers
%Lie groupoids (as in \cite{Connes:1994}), the word discrete usually designates 
%groupoids whose underlying topological space is discrete.)
{\em contractible} if it is equivalent to a singleton set.
%Equivalently, a groupoid is contractible if its underlying graph is a complete
%directed graph.
\end{blanko}

% \begin{blanko}{Fibrations of groupoids.}
%   A morphism of groupoids $p:X\to Y$ is a {\em fibration} if it has the path lifting
%   property: for each object $x$ of $X$ and arrow $g:y'\to px$ of $Y$ there
%   exists an arrow $f:x'\to x$ such that $pf=g$.
% %   Fibrations are sometimes called
% %   \emph{star surjective} functors since they induce surjections
% %   $\hom_X(-,x)\onto\hom_Y(-,px)$, where  $\hom_x(-,x)$ denotes the set of all
% %   arrows with codomain $x$, termed the \emph{star} of $x$ in $X$.
%   Fibrations are really just a technical notion to simplify some constructions.
%    We will see below that any morphism may be replaced by a fibration if
%   necessary.
% \end{blanko}

% \begin{blanko}{Sums of fibres.}
%   A basic construction in combinatorics is to split a set into a disjoint union
%   of parts: given a map of sets $E\to B$, the `total space' $E$ is the sum of the
%   fibres:
%   $$
%   E = \sum_{b\in B} E_b .
%   $$
%   The same formula holds for groupoids, provided we use homotopy fibres and 
%   homotopy sums (denoted by integrals).  In this case the formula says: given a map of groupoids
%   $E\to B$, we have a natural equivalence
%   $$
%   E \simeq \grint{b}{B}{E_b} .
%   $$
%   This viewpoint is fundamental to our work, and
%   we need to explain these two notions.  Homotopy fibre is a special case of
%   homotopy pullback, and homotopy sum is a special case of homotopy colimit.
% \end{blanko}

\begin{blanko}{Pullbacks.}  
%  REFER TO Hakim~\cite{Hakim} AND Brown-Heath-Kamps~\cite{
%  \nocite{Brown:topology-and-groupoids}
%  The na\"ive notions of pullback and fibres are not very useful for groupoids, 
%  as these notions are not stable under equivalence.  The appropriate notions
%  are homotopy pullbacks and homotopy fibres.
Recall that the {\em homotopy pullback}~\cite{Brown-Heath-Kamps:Toledo}
or {\em fibre product}
of a diagram of groupoids 
$$
X\xrightarrow {\quad g\quad} S\xleftarrow{\quad f\quad} Y
$$
%(usually called a {\em cospan}) 
is the groupoid $X\times_S Y$ whose objects 
are triples
$(x,y,\phi)$
 with $x \in X$, $y\in Y$ and  $\phi:fx\to gy$ an arrow of $S$, and whose
 arrows are pairs $(\alpha,\beta):(x,y,\phi)\to(x',y',\phi')$ consisting
 of
 $\alpha: x \to x'$ an arrow in $X$ and $\beta: y \to y'$ an arrow in $Y$
 such that $g(\beta)\phi=\phi'f(\alpha):fx\to gy'$.

Following~\cite[2.6.2]{Hakim}, one can say that the homotopy pullback and the projections to $X$ and $Y$ can be characterised up to canonical equivalence by a universal 
property: it is the $2$-terminal object in a category of diagrams of the form
$$
\xymatrix{
W \ar[r]\ar[d]\ar@{}[rd]|{\simeq}
&Y\ar[d]^-{g}\\X\ar[r]_-{f}&S
}
$$
where  $2$-terminal means that the comparison
map is not unique but rather that the comparison maps form a contractible groupoid.

% If $f$  is a fibration then so is $q$, and in
% this case the homotopy pullback is equivalent to the na\"\i ve (strict) pullback.
% The \emph{fibrant replacement} of a
% morphism $p:E\to B$ is an equivalent fibration $\bar p: \bar E \to B$, which can
% be obtained by performing the above explicit (homotopy) pullback construction of
% $p$ along the identity morphism $B\to B$.
% Indeed, for any object $(b',e,b'\stackrel\phi\longrightarrow
% pe)$ of $\bar E$, any arrow $g:b\to b'$ may be lifted to an arrow in $\bar
% E$,
% $$
% \xymatrix@C=3ex{
% b\ar[r]^-{\phi g}\ar[d]_-{g}&pe\ar@{=}[d]\\
% b'\ar[r]_-{\phi}&pe
% .}
% $$ 
\end{blanko}
\begin{blanko}{Fibres.}
The notion of fibre is a special case of pullback, and again we need the 
homotopy version. 
The \emph{homotopy fibre} $E_b$ of a morphism $p:E\to B$ over an object $b$ in $B$ is 
the following homotopy pullback:
$$
\xymatrix{
E_b\drpullback\ar[r]\ar[d]
&E\ar[d]^-{p}\\1\ar[r]_-{\name b}&B.
}
$$
Here ${\name b}:1\to B$ is the inclusion morphism, termed the {\em name} of $b$.

For a morphism $b\to b'$ there is a canonical functor $E_b\to E_{b'}$,
as one sees from the explicit description of fibres as pairs $(e,\phi:pe\cong b)$ and their isomorphisms. Thus we have a strict functor
\begin{equation}\label{fibrefunctor}
F:B\to\Grpd,\qquad F(b)= E_b.
\end{equation}

% Again, if $p$ is a fibration then the homotopy fibre is equivalent to the strict fibre.  
Henceforth the words pullback and fibre will always mean the homotopy
pullback and homotopy fibre, since these notions are invariant under
equivalence (unlike the strict notions).

%For a  fibre product of
%groupoids $X\times_J Y$, the fibre $(X\times_J Y)_j$ is naturally
%equivalent to the ordinary product $X_j \times Y_j$.  There are 
%actually two naturally equivalent ways of defining the fibre $(X\times_J Y)_j$, since t%here are two distinct 
%but homotopic ways to go around the square.

% The \emph{slice} $X_{/x}$ over an object $x$ of $X$ is the fibre $X_x$
% of  the identity functor $X\to X$.
% IS THIS NEEDED AT ALL? IF SO: HERE, OR IN \ref{subsect-slices}?

%, so 
% $$
% \xymatrix{X/x\ar[r]^-p\ar[d]
% &X\ar[d]^-{=}\\1\ar[r]_-{\name x}&X}
% $$
%the functor $X_{/x}\to X$ is always a fibration.
% NO LONGER NEEDED, PROVED MORE GENERALLY ABOVE
%: for any object $y'\stackrel\phi\longrightarrow x$ of $X/x$, all arrows $g:y\t%o y'$ may be lifted to arrows
%$$ 
%\xymatrix{
%y\ar[r]^-{\phi g}\ar[d]_-{g}&x\ar[d]^-{=}\\
%y'\ar[r]^-{\phi}&x
%}
%$$ 
%in $X/x$. 
% In general, the
%
%The
% fibre $E_b$ will rarely be equivalent to the strict pullback of $p:E\to B$
% and $\name b:1\to B$ unless $p$ is a fibration, but we may replace the functor
% $\name b$ by the fibration $B_{/b}\to B$ and take the strict pullback.

\end{blanko}

%  To explain homotopy sums, it is useful to start with homotopy quotients.

\begin{blanko}{Homotopy quotient.}
%   When a $G$ group acts on a set $X$,  by adding between objects $x$ and $y$ an
%   arrow for each $g\in G$ such that $xg=y$.  
%   
%  Whenever a group $G$ acts on a groupoid $X$, the {\em homotopy quotient}
%  $X/G$, sometimes called the action or orbit groupoid, semidirect product, or
%  weak quotient (and often denoted $X/\!/G$ in the literature), is the groupoid
%  described as follows.  
Whenever a group $G$ acts on a groupoid $X$, 
the {\em homotopy quotient} $X/G$ 
(often denoted $X/\!/G$, 
and known also as orbit groupoid 
\cite[Ch.~11]{Brown:topology-and-groupoids}, 
semi-direct product \cite{Brown-Heath-Kamps:Toledo}, 
\cite[II.5]{Connes:1994}, and
weak quotient \cite{Baez-Dolan:finset-feynman},
\cite{Baez-Hoffnung-Walker:0908.4305}) 
is the groupoid described as follows.
Its objects are those of $X$.  An arrow in $X/G$ from
  $x$ to $y$ is a pair $(g,\phi)$ with $g\in G$ and $\phi: x.g \to y$ an arrow
  in $X$.  Intuitively, $X/G$ is obtained from $X$ by sewing in a path in $X$
  for each object $x$ and each (non-identity) element of the group.  If $X=1$, a
  singleton, then $1/G$ is the groupoid with a single object $1$ and vertex
  group $G$.  This groupoid may also be denoted $BG$, analogous to the classifying space
  in topology.
\end{blanko}

\begin{blanko}{Grothendieck construction and homotopy sum.}\label{hosum}
A family of sets indexed by a set $B$ can be described 
either as a map 
$f:E \to B$ (the members of the family are the fibres $E_b := f^{-1}(b)$) 
or as a map 
$F:B \to \Set$ (the members are then the values $F(b)$).
Similarly, as we proceed to recall, a family of groupoids indexed by a groupoid $B$ can be described in two equivalent ways: 
either as a functor $B \to \Grpd$, or as a map of groupoids $E \to B$. 

Given a functor $F: B \to \Grpd$, the {\em Grothendieck
  construction} (see SGA1~\cite{SGA1}, Exp.VI, \S 8)
  produces a new groupoid $E$ (actually the
  homotopy colimit of $F$) together with a map $E \to B$.  The objects of $E$
  are pairs $(b,x)$ where $b\in B$ and $x\in F(b)$; an arrow from $(b,x)$ to
  $(b',x')$ is a pair $(\sigma,\phi)$ where $\sigma: b \to b'$ is an arrow of
  $B$, and $\phi: (F\sigma)(x) \to x'$ is an arrow of $F(b')$.
  The map $E \to B$ is the projection.
  The groupoid $E$ is called the {\em homotopy sum} of the family $F$,
  and is denoted
  $$
  \grint{b}{B}{F(b)} .
  $$
  
  % On the other hand, every map of groupoids $E \to B$ gives rise to a
  % functor $B \to \Grpd$, sending $b$ to the homotopy fibre $E_b$.
  % (Note that this
  % is a strict functor, due to the fact that we are using homotopy fibres;
  % in the classical setting of Grothendieck the groupoid-valued  functor
  % associated to a groupoid fibration is only a pseudo-functor.)

This construction is mutually inverse to the construction in \eqref{fibrefunctor} of the functor $F:B\to\Grpd$ from a map $E\to B$.

%   These
%   two constructions are mutually inverse in the sense that the 
%   Grothendieck construction of the functor associated to a map $E \to B$ is
%   canonically equivalent to the original map, and that the functor associated
%   to the Grothendieck construction of a functor is canonically equivalent to
%   that original functor.
  
% %   {\em homotopy 
% %   colimit} of $F$ can be defined as a homotopy initial object in a certain
% %   BLA BLA.  
  
%   The constructions preserve equivalences and are compatible with disjoint
%   unions, so that the constructions can be performed separately on each
%   connected component of $B$.  To describe the constructions, it is therefore
%   enough to treat the case $B=BG$, the groupoid with only one
%   object $*$ and vertex group $G$.  The constructions now boil down to group
%   actions and weak quotients.  Namely, to give a functor $F:BG \to \Grpd$ is the
%   same as giving a group action $X \times G \to X$, where $X$ is defined as
%   $X=F(*)$.  The Grothendieck construction is now precisely the weak quotient,
%   according to the explicit description given above.  Conversely, given any map
%   of groupoids $E \to BG$, there is a canonical $G$-action on the homotopy fibre
%   over the unique point.  This is immediate from the explicit description
%   of homotopy fibre.
  
\end{blanko}

% The upshot of all this is the following.

\begin{prop}\label{prop:groth}
  Given a map of groupoids $f:E \to B$, the total space $E$ is equivalent (over $B$) to the
  homotopy sum of its fibres:
  $$
  E \simeq  \grint{b}{B}{E_b}.
  $$
This can be computed (up to equivalence) as
  $$
  E\simeq \sum_{b\in \pi_0 B} E_b / \Aut(b) .
  $$
\end{prop}
\begin{proof}
This is straightforward: one checks that the explicit construction of the homotopy pullback of $f$ along id$_B$ is actually isomorphic to the Grothendieck construction of the functor \eqref{fibrefunctor}.
%One may replace the map $f$ by a fibration of groupoids whose total space is equivalent to $E$ and whose
%strict fibres are the homotopy fibres of $f$, and we have the classical result
% that the Grothendieck construction of the fibre functor is also equivalent to $E$.\  
Applying the first result to the composite $E\to B\isopil \sum_{b\in\pi_0 B}1/\Aut(b)$
 %, where one replaces $B$ by its skeleton (see \cite[p.91]{MacLane:categories}), 
one obtains the second.
\end{proof}

The following can be seen as a Fubini lemma:
\begin{lemma}
%   (Transitivity of integration)
  Given morphisms of groupoids $X \stackrel f \to B \stackrel t \to I$, we have
  $$
  \int^{b\in B} X_b  \;\;\simeq \;\;\int^{i\in I} \left( \int^{b\in B_i} X_b  
  \right)  
  $$
  over $I$.
\end{lemma}
Again, the proof of the lemma is straightforward, yet it
automatically takes care of a lot of automorphism yoga which
  without the setting of groupoids tends to become messy.
  Already spelling it out in (set) sums and group actions reveals 
  that a lot is going on:
The formula says  
    $$
  \sum_{b\in \pi_0 B} X_b /\Aut(b) \simeq 
   \sum_{i\in \pi_0 I} \left(\sum_{b\in \pi_0 B_i} X_b /\Aut_i(b)\right) / 
   \Aut(i).
  $$
  Note
that $\pi_0 B_i$ denotes the set of connected components of the fibre $B_i$ which is
  typically different from the set of connected components of $B$ that intersect
  the fibre: objects in the fibre might be connected only via arrows in $B$
  that are not in the fibre. Similarly,
 $\Aut_i(b)$ denotes the vertex group of $b$ in the fibre
  $B_i$, not the whole vertex group $\Aut(b)$.% of $B$. 

Applying Proposition~\ref{prop:groth} twice we get the following easy double-counting lemma.  It
can be seen as the groupoid analogue of double counting in a bipartite
graph, held by Aigner~\cite{Aigner} as one of the most important principles in
enumerative combinatorics.

  \begin{lemma}\label{doublecounting}
Let $A,B,U$ be groupoids, together with morphisms
$$\xymatrix{B&\ar[l]U\ar[r]&A}$$
and write ${}U_S,\: _TU\subseteq U$ for the (homotopy) fibres over $S\in A$ and $T\in
B$ respectively.  Then there are equivalences of groupoids
\begin{align*}
  \grint{T}{B}{{}_TU}
&\;\;\simeq \;\;U\;\;\simeq\;\;
\grint{S}{A}{U_S}
.
\end{align*}
\end{lemma}

% FLAGGED FOR DELETION:
% Integration is the analogue of counting: 
% counting is establishing any non-canonical isomorphism with one of the standard
% sets $n\in \N$.  Taking skeleton, or other integration formulae, is similar:
% just to get a size estimate. It is therefore typically a last step in a 
% calculation, but it is abusive to do it all the time.

\begin{blanko}{Slices.}\label{subsect-slices}\label{adj}
  We shall need homotopy slices, sometimes called weak slices. 
% First recall the usual notion of slice category:
%   If $\CC$ is a category, and $I\in \CC$, then the
%   usual slice category $\CC_{/I}$ is the category whose objects are morphisms $X \to
%   I$ in $\CC$ and whose arrows are commutative triangles
% $$\xymatrix@R-2ex{%begin{diagram}[w=2.7ex,h=4.5ex,tight] %trianglediagram
% X \ar[rdd] \ar[rr]    && X' \ar[ldd]   \\ \\
% &     I   .
% }$$
%Consider a groupoid-enriched category $\CC$, i.e.~a category in which
%for each pair of objects $X,Y$, the arrows $\Hom(X,Y)$ form a 
%groupoid instead of just a set, and such that composition is a functor
%instead of just a function.
%Thus,  between two parallel arrows $X \rightrightarrows Y$, there may be
%invertible $2$-cells. The basic example is the groupoid-enriched category 
%The category $\Grpd$ has 2-cells, which are the natural isomorphisms.
%all groupoids: the objects are groupoids, the morphisms are functors, and the
%$2$-cells are natural transformations (automatically invertible).
 The {\em slice category} $\Grpd_{/I}$ 
has as objects the morphisms $X \to I$, and arrows
are triangles with a $2$-cell, that is, a natural transformation:
\begin{equation}\label{triangledigram}
\vcenter{\xymatrix@R-3.5ex{%begin{diagram}[w=2.7ex,h=4.5ex,tight] %trianglediagram
X \ar[rdd] \ar[rr]    && X' \ar[ldd]   \\ &\Rightarrow\\
&     I   .
}}
\end{equation}
Arrows are composed by pasting such triangles.
%, and altogether
%the slice is a category (in fact again a groupoid-enriched category, but this 
%will not be important here).
%If $\CC$ is a groupoid then  $\CC_{/I}$ is again a groupoid.

% In this setting, the notion of slice is weakened by allowing the triangles to
% commute only up to a specified $2$-cell.  Hence if $I$ is a groupoid,
% the slice category $\Grpd/I$ has as objects the maps $X \to I$; its arrows
% are triangles with a $2$-cell.

%\end{blanko}

%\begin{blanko}{Adjoints between slices.}

Taking homotopy pullback along a morphism of groupoids  $f:B'\to B$  defines a functor
$$
f\upperstar
:\Grpd_{/B}\to \Grpd_{/B'} .
$$
This has a homotopy left adjoint, defined by composition with $f$,
$$f\lowershriek  :\Grpd_{/B'}\to \Grpd_{/B}$$ 
and a homotopy right adjoint 
$$f\lowerstar :\Grpd_{/B'}\to \Grpd_{/B},$$ 
in the sense that there are
%The homotopy adjoint properties are expressed by 
natural equivalences of mapping groupoids
\begin{align}
\Grpd_{/B}(f\lowershriek E',E)&\simeq
\Grpd_{/B'}(E',f\upperstar E),\label{slice-adj}
\\
\Grpd_{/B'}(f\upperstar E,E')&\simeq
\Grpd_{/B}(E,f\lowerstar E') .
\end{align}
%which will be invoked at a few occasions.
\end{blanko}

\begin{blanko}{$I$-coloured finite 
  sets, or families of objects in $I$.}\label{fam}
 Let  $\Bij$ denote the groupoid of finite sets and bijections. Since a set may be regarded as discrete groupoid, we can consider $\Bij$ as a groupoid-enriched subcategory of $\Grpd$.
For a groupoid $I$, the groupoid of $I$-{\em coloured sets} is the slice category  $$
  \fam I := \Bij_{/I} .
  $$
%   
%   
%   (Since $I$ is a groupoid, not a set, to make formal sense of this,
% we put $\Bij/I := D\comma \name{I}$, with reference to the comma diagram
% \begin{diagram}
% D\comma \name{I} \SEpbk& \rTo & \Grpd/I & \rTo &  1 \\
% \dTo && \dTo & \Rightarrow & \dTo>{\name{I}} \\
% \Bij & \rTo_D & \Grpd & \rTo_{\Id} & \Grpd
% \end{diagram}
% is this really needed?)
%   
Hence an $I$-coloured set is a
  groupoid morphism $X \to I$, where $X$ is a finite set, and isomorphisms between them are are triangles with a $2$-cell as in \eqref{triangledigram}.
%  The groupoid $\fam I$ can be interpreted as the groupoid of $I$-coloured
%  finite sets: the map $X \to I$ then associates a `colour' to each element in $X$.
%  Note that maps of $I$-coloured sets are required to be bijective and respect the 
%  colour,
%  but only up to specified isomorphism of colours (that's the content of the
%  $2$-cell triangle).   
If $I=1$ is the one-point trivial groupoid, we recover the groupoid of (one-coloured) sets and bijections,
  $\fam 1 \simeq \Bij$.

  The groupoid $\fam I$ can be considered also as the groupoid of families of
  objects in $I$.  In this case, the finite set $X$ plays a secondary role, it
  is merely an indexing set for the family.  We use this viewpoint for example
  when we say that a forest is a family of trees.  Formally, if $\tr$ is the 
  groupoid of trees (cf.~below), then the groupoid of forests is
  $$
  \forest = \fam \tr .
  $$
  
%  As another important example, note that a surjection of sets is just a
%  disjoint union of connected surjections, and a connected surjection is
%  determined by a single non-empty set (mapping to a point), so the groupoid
%  of surjections can be considered as the groupoid of families of non-empty
%  finite sets and bijections,
%  $$
%  \surj = \fam \nonempt .
%  $$
  
  It should be mentioned, although we will not need this fact, that $\fam I$ is 
  the free symmetric monoidal category on $I$.
 \end{blanko}

%We finish this section by recalling a notion which concerns categories rather 
%than just groupoids, but which will be needed later.

%%%%%%%%%%%%%%%%%%%%%%%%%%%%%%%%%%%%%%%%%%%%%%%%%%
\section{Trees and forests}
%%%%%%%%%%%%%%%%%%%%%%%%%%%%%%%%%%%%%%%%%%%%%%%%%%
\label{sec:trees}

\begin{blanko}{Polynomial functors.}
%  The importance of the above tree formalism 
%  is that diagrams of shape \eqref{tree} are precisely what define polynomial endofunctors.
  The theory of polynomial functors (for which we refer to 
  \cite{Gambino-Kock:0906.4931}) is very useful to encode combinatorial 
  structures, types and operations, and covers notions such as species and operads.
Any diagram of groupoids 
 $$
 I \stackrel s \lTo E \stackrel p \rTo B \stackrel t \rTo I
 $$
 defines a {\em polynomial endofunctor} as the composite  (see \ref{adj})
 $$
 \Grpd_{/I} \xrightarrow {s^* } \Grpd_{/E} \xrightarrow {p_* }
 \Grpd_{/B} \xrightarrow{t\lowershriek}  \Grpd_{/I} .
 $$
% (Here of course we are talking about homotopy slices, and upperstar, lowerstar 
% and lowershriek refer to 
%.)
% In this work we do not need the actual 
% functors, only their representing 
%  diagrams.  
The intuition is that $B$ is a collection of typed operations. The arity of an operation $b$ is given by the size of the fibre $E_b$, the input types are the $s(e)$ for $e\in E_b$, and the output type is $t(b)$.
  
We shall see examples of polynomial functors
  in Section~\ref{sec:ex}.

\end{blanko}

\begin{blanko}{Trees.}\label{polytree-def}
  It was observed in \cite{Kock:0807} that operadic trees can be
  conveniently encoded by diagrams of the same shape as polynomial functors. 
  By definition, a
  {\em tree} is a diagram of finite sets
\begin{equation}\label{tree}
\xymatrix{
    A & \ar[l]_s  M  \ar[r]^p & N  \ar[r]^t & A
}
\end{equation}
satisfying the following three conditions:
  
  (1) $t$ is injective
  
  (2) $s$ is injective with singleton complement (called the {\em 
  root} and denoted $1$).
  
  \noindent With $A=1+M$, 
  define the walk-to-the-root function
  $\sigma: A \to A$ by $1\mapsto 1$ and $e\mapsto t(p(e))$ for
  $e\in M$. 
  
  (3)  $\forall x\in A : \exists k\in \N : \sigma^{k}(x)=1$.
  
  The elements of $A$ are called {\em edges}.  The elements of $N$
  are called {\em nodes}.  For $b\in N$, the edge $t(b)$ is called
  the {\em output edge} of the node.  That $t$ is injective is just to
  say that each edge is the output edge of at most one node.  For
  $b\in N$, the elements of the fibre $M_b:= p^{-1}(b)$ are
  called {\em input edges} of $b$.  Hence the whole set
  $M=\sum_{b\in N} M_b$ can be thought of as the set of
  nodes-with-a-marked-input-edge, i.e.~pairs $(b,e)$ where $b$ is a
  node and $e$ is an input edge of $b$.  The map $s$ returns the
  marked edge.  Condition (2) says that every edge is the input edge
  of a unique node, except the root edge.
  Condition (3) says that if you walk towards the root, in a finite 
  number of steps you arrive there.
  The edges not in the image of $t$ are called {\em leaves}.
%  From now on we just say {\em tree} for `operadic tree'. 
  The tree
  $
  1 \leftarrow 0 \to 0 \to 1
  $
  is the {\em trivial tree} \inlineDotlessTree.
\end{blanko}

\begin{blanko}{Morphisms of trees (cf.~\cite{Kock:0807}).}\label{sub}
  A {\em tree embedding} is by definition a diagram
  \begin{equation}
    \label{equ:cartmorphism}
\vcenter{\xymatrix{
  A' \ar[d]_\alpha& \ar[l]  M'\ar[d] \drpullback \ar[r] & N'\ar[d] \ar[r] & A'\ar[d]^\alpha  \\
  A  &\ar[l] M\ar[r] & N \ar[r]  &A ,
  }}
  \end{equation}
  where the rows are trees.
%  This is just the notion of cartesian morphism in the category of 
%  polynomial endofunctors \cite{Gambino-Kock:0906.4931}.
%, cf.~Section~\ref{sec:poly} below.
%  The terminology is justified by the fact that each of
%  the components of such a map is necessarily injective; this follows from the 
%  tree axioms 
%.  Hence the category 
%, is mostly concerned with subtrees,
%  but note that it also contains automorphisms of trees.
%  
(It follows from the tree axioms that the components are injective.)
 The fact that the middle square is cartesian means that there is
  specified, for each node $b$ of the first tree, a bijection between
  the incoming edges of $b$ and the incoming edges of the image of $b$.
  In other words, a tree embedding is arity preserving.

  A tree embedding is {\em root-preserving} when it sends the root to the root.
  In formal terms, these are diagrams \eqref{equ:cartmorphism}
  such that also the left-hand square is cartesian.
  
  An {\em ideal embedding} (or an {\em ideal subtree}) is a subtree $S$
  in which for every edge $e$, all the edges and nodes above $e$ are also in $S$.
  There is one ideal subtree
  generated by each edge in the tree.  The ideal embeddings are characterised
  as having also the right-hand square of \eqref{equ:cartmorphism} cartesian.
  
  Ideal embeddings and root-preserving embeddings admit pushouts along each 
  other in the category $\TEmb$ 
of
  trees and tree embeddings 
\cite{Kock:0807}.
  The most interesting case is pushout over a 
  trivial tree: this is then the root of one tree and a leaf of another tree,
  and the pushout is the grafting onto that leaf.
\end{blanko}

\begin{blanko}{Decorated trees: $P$-trees.}\label{Ptree}
  An efficient way of encoding and manipulating decorations of trees
  is in terms of polynomial functors \cite{Kock:0807} (see also 
  \cite{Kock:1109.5785,Kock:MFPS28,Kock:graphs-and-trees,Kock-Joyal-Batanin-Mascari:0706}).
  %, which we shall review in   Section~\ref{sec:poly}.  
Given a polynomial
  endofunctor $P$ represented by a diagram
  $
  I \stackrel s \lTo E \stackrel p \rTo B \stackrel t \rTo I ,
  $
  %We further require $p$ to have finite and discrete fibres (i.e.~the
  %fibres are equivalent to finite sets).  
%  which we  keep fixed throughout, until in Section~\ref{sec:ex} 
%  where we consider different choices for $P$.
%  By definition, 
a {\em $P$-tree} is a diagram
$$    \xymatrix{
  A \ar[d]_\alpha& \ar[l]  M\ar[d] \drpullback \ar[r] & N\ar[d] \ar[r] 
& A\ar[d]^\alpha \\
  I  &\ar[l] E\ar[r] & B \ar[r]  &I ,
  }$$
  where the top row is a tree.  
  The squares
  are commutative up to isomorphism, and it is important that the isos 
  be specified as
  part of the structure.
  Unfolding the definition, we see that a $P$-tree is a
  tree whose edges are decorated in $I$, whose nodes are decorated
  in $B$, and with the additional structure of a bijection for each
  node $n \in N$ (with decoration $b \in B$) between the set of
  input edges of $n$ and the fibre $E_b$, subject to the
  compatibility condition that such an edge $e\in E_b$ has
  decoration $s(e)$, and the output edge of $n$ has decoration isomorphic
  to $t(b)$.
  
  Standard examples of $P$-trees are given in Section~\ref{sec:ex},
  where we also consider groupoid-polynomial 
  decorated trees arising  naturally in quantum field theory, where in
  order to account for symmetries it is crucial that the representing
  diagram $I \leftarrow E \to B \to I$ be of groupoids, not just sets.
% polynomial for decorating trees is
% $$
% I \leftarrow E \to B \to I
% $$
% where $I$ is the groupoid of interaction labels, $B$ is the groupoid
% of connected 1PI graphs of the theory, and $E$ is the groupoid of
% such graphs with a marked vertex.  In order to handle correctly the
% insertion of one graph into the vertex of another, the symmetries of
% the lines of the vertex are essential.  Hence the edge colours of
% the relevant trees take values in a groupoid, not just in a set.

  The category of $P$-trees is the slice category $\TEmb_{/P}$.
  The notions of root-preserving and ideal embeddings work the same in this
  category as in $\TEmb$, and again these two classes of maps allow pushouts
  along each other.
%Observe that $P$-trees can have more
%automorphisms than the underlying tree.  For example, if $P$ is given by
%$I \leftarrow 1 \to 1 \to I$, where
%the groupoid
%$I$ has one object and vertex group $G$, then the trivial $P$-tree
%$\inlineDotlessTree$ has also automorphism group $G$.  This follows
%easily from the observation that the $I$-family $ 1 \to I$ has
%automorphism group $G$.
\end{blanko}

%\begin{blanko}{The bialgebra of $P$-trees.}
%  This is precisely the same prescription as for naked trees in 
%  \ref{treebialg}.
%\end{blanko}

\begin{blanko}{Forests.}\label{forests}
  A forest can be defined as a family of trees, or equivalently as a finite sum
  of trees in the category of polynomial endofunctors.  It is convenient to
  have also an elementary definition, similar
  to that of trees.

  By definition, a
  {\em (finite rooted) forest} is a diagram of finite sets
$$\xymatrix{
    A & \ar[l]_s  M  \ar[r]^p & N \ar[r]^t & A
}$$
satisfying the following three conditions:
  
  (1) $t$ is injective
  
  (2) $s$ is injective; denote its complement $R$ (the set of roots).
  
  \noindent With $A=R+M$, 
  define the walk-to-the-roots function
  $\sigma: A \to A$ by being the identity on $R$, and $e\mapsto t(p(e))$ for
  $e\in M$. 
  
  (3)  $\forall x\in A : \exists k\in \N : \sigma^{k}(x)\in R$.
  
  The interpretations of these axioms are similar to those following the 
  definition of tree.
  
  A {\em forest embedding} is by definition a diagram like 
  \eqref{equ:cartmorphism}, required now separately to be injective (whereas for
  trees this condition is automatic, for forests absence of the condition gives 
  only etale maps). 
%   In this case, the maps cannot be guaranteed to
%   be injective, but they are locally injective in each connected component,
%   and together with the
%   cartesian property ensures that the map is locally an isomorphism, whence
%   the name `etale'.  For example, the forest consisting of isomorphic trees
%   can form a double cover of a single copy of that tree.
  
  A forest embedding is called a {\em root-preserving embedding} if it induces
  a bijection between the sets of roots.  This is equivalent to being a sum of
  tree embeddings.  By {\em ideal embedding} we understand an embedding
  such that the right-hand square of \eqref{equ:cartmorphism} is cartesian.
  This means that each edge and node above the subforest is also
  contained in the subforest. 
  The most important example will be this: for a 
  given tree $S$, the set of its leaves forms a forest, and the inclusion of this
  forest into $S$ is an ideal embedding.
  
  Just as for trees, root-preserving embeddings and ideal embeddings allow
  pushouts along each other (in the category of forests and forest embeddings).
  The important case is grafting a forest onto the leaves of a tree.
\end{blanko}

\begin{blanko}{$P$-forests.}
    The definition of $P$-forest is analogous to the definition of 
    $P$-tree, and again the category of $P$-forest embeddings can be 
    characterised as the finite-sum completion of $\TEmb_{/P}$ inside the slice 
    category $\kat{Poly}_{/P}$.
\end{blanko}

%We fix a polynomial endofunctor $P$ (given by $I \leftarrow E \to B \to I$)
%and denote by $\tr$ the groupoid of $P$-trees and by $\forest$ the groupoid of 
%$P$-forests.

%%%%%%%%%%%%%%%%%%%%%%%%%%%%%%%%%%%%%%%%%%%%%%%%%%
\section{\fdb equivalence in the groupoid of trees}
%%%%%%%%%%%%%%%%%%%%%%%%%%%%%%%%%%%%%%%%%%%%%%%%%%
\label{sec:fdbequiv}

We fix a polynomial endofunctor $P$ given by
$$
I \leftarrow E \to B \to I  .
$$
% such that $E\to B$ has finite discrete fibres.
Throughout this section 
the word `tree' will mean $P$-tree, and `forest' will mean $P$-forest.
% for a finitary polynomial endofunctor
% The finer structure of $P$-trees plays no role: the only
% thing that matters here is that edges are coloured in the groupoid $I$.
We denote the groupoids of $P$-trees and $P$-forests by $\tr$ and $\forest$ respectively.

In this section we  prove our main theorem, 
the equivalence of groupoids over $\forest\times\tr$
$$
\grint{T}{\tr}{\cut(T)} \simeq \grint{\leaves{N}}{\fam I}{
%\Grpd/I(\leaves{n},\tr) 
  \forest_{N}
\times {}_{\leaves{N}}\tr} .
$$
In Section~\ref{fdbalg} 
we will
obtain the \fdb formula 
for the Green function in the bialgebra of trees
by taking relative cardinality of both sides.

% IS THE FOLLOWING PARAGRAPH NECESSARY?

\begin{blanko}{Leaves and roots.}
  To any tree or to any forest we can associate its set of leaves.  These are 
  naturally  $I$-coloured sets, and we have  groupoid morphisms  
$
L\colon \tr \to \fam I$ and $
L\colon \forest \to \fam I
$,
called the {\em leaf maps}. Similarly, taking the root of a tree, and the set of roots of a forest, we have groupoid morphisms
$R\colon\tr \to I$ and 
$R\colon\forest\to \fam I$,
%An object in $\fam I$ can be interpreted as a {\em leaf profile},
%and we can ask for those trees with a given leaf profile $N\in \fam I$.
%This is the homotopy  fibre
%$$\xymatrix{
%    {}_{\leaves{N}}\tr\drpullback \ar[d]\ar[r] & \tr \ar[d]\\
%    1  \ar[r]_{\name N} & \fam I.
%}$$
% Note that inside the fibre ${}_N \tr$ the automorphism group of an
% object really consists of those auts of the tree that fix the leaves
% completely.  Indeed, an object in the fibre is a pair
% $(T,\;\phi\!:\!LT \simeq N)$ where $\phi$ is an isomorphism $g$ 
% \marginpar{g??}
% in $\fam I$.
% This means a bijection of sets $LT\simeq N$ equipped with an isomomorphism
% in $I$ to match up the colours.  An automorphism of such an object
% is an automorphism of the tree $a:T\simeq T$, such that $La$ commutes
% with the given $\phi$ and $g$.  This is only possible when $La$ is the
% identity, in other words, $a$ fixes the leaves completely (for example
% by permuting some nullary nodes, or by automorphisms coming from the
% decorations elsewhere than at the leaves).
called the {\em root maps}.
We use two-sided subscript notation to indicate the fibres of these maps,
\begin{equation}\label{span}
\vcenter{\xymatrix{&\ar[ld]_L\forest\ar[rd]^R&&\ar[ld]_L\tr\ar[rd]^R\\
{\fam I}&&{\fam I}&&I}}
\end{equation}
Hence, we denote by $\tr_k$ the groupoid of trees with root colour $k\in I$
(or more precisely: with root colour isomorphic to $k$, and with a specified 
iso)
and by $\forest_N$ the groupoid of forests whose set of roots is $N\in {\fam 
I}$ (again, up to a specified iso). 
Similarly, for the fibres of $L$, we write ${}_{\leaves{N}}\forest$ and 
${}_{\leaves{N}}\tr$ 
for the groupoids of forests and  
trees with leaf profile $\leaves{N}$.
These are the groupoids
of $P$-forests or $P$-trees with specified $I$-bijections between their leaves and 
$\leaves{N}$.

%Let $\forest$ denote the groupoid of $P$-forests.  
%There is also a natural 
%morphism, the {\em root map}
%$$
%which to a forest associates its $I$-coloured set of roots.  
%For a 
%fixed colour $\leaves N : X \to I$,
% the $\leaves{N}$-fibre,
%$$\xymatrix{
%    \forest_{N}\drpullback \ar[d] \ar[r]& \forest\ar[d] \\
%    1 \ar[r]_{\name N} & \fam I
%}$$
%has the following characterisation:
\end{blanko}

The groupoid of forests with a given {\em root profile} has the following characterisation:
\begin{lemma}\label{Fn=T^n}
    $$
    \forest_N \simeq \Grpd_{/I}(\leaves{N},\tr) .
    $$
\end{lemma}

\begin{dem}
%  Recall that 
The forest root map $\forest\to\fam I$ is the family functor applied to 
  the tree root map, that is, $\fam R : \fam \tr \to \fam I$.
  Hence we can write, by adjunction: 
$$
\forest_N \simeq \name{N}\upperstar \fam R \simeq \Grpd( 1, \name{N}\upperstar \fam R)
\simeq \Grpd_{/\fam{I}} ( \name{N}, \fam R).
$$
It remains to establish the equivalence
$$
\Grpd_{/\fam{I}} ( \name{N}, \fam R)
\simeq \Grpd_{/I} (N, R) .
$$
% 
% Now by the following guess, tilde is right adjoint to colimits, 
% and the colimit of $\name{N} : 1 \to \fam I$ is the family $N : E \to 
% I$.  So we now find
% $$= \Grpd_{/I} (N, R) = \Grpd_{/I}(E, \tr)
% $$
% as claimed. 
% 
% \bigskip
% 
% The adjunction we need is
% $$\Grpd_{/I} ( \colim(a),F) = \Grpd_{/\fam I}( a, \fam F)
% $$
% Instead of proving the adjunction we should just prove by hand that
% $$\Grpd_{/\fam I}(1,\fam R) = \Grpd_{/I} (E, R)
% $$
% Here is the proof:
Consider the commutative diagram
$$\xymatrix{
\Grpd_{/I} (N, R) \ar[r]\ar[d] & \Grpd_{/\fam I}(\name{N},\fam R) \ar[d] \\
\Grpd (X, \tr) \drpullback\ar[r]\ar[d] & \Grpd(1,\fam \tr) \ar[d] \\
\Grpd(X, I) \ar[r] & \Grpd(1,\fam I)
}
$$ in which the vertical maps form the standard slice fibre sequences;
the bottom vertical maps are postcomposition with $R$ and $\fam R$, 
respectively.
Each of the horizontal maps sends a family to its name.
Since the bottom square is a pullback, we conclude that the top map
is an equivalence.
\end{dem}

%Integrating over the fibres, we therefore find
%$$
%\forest \simeq \grint{N}{\fam I}{\forest_{N}}
%\simeq \grint{N}{\fam I}{\Grpd_{/I}(\leaves{N},\tr)} \simeq \exp(\tr) ,
%$$
%corresponding again to the fact that forests are disjoint unions of 
%trees.

\begin{blanko}{The groupoid of trees with a cut.}
  In~\ref{treebialg} we already defined a cut in a tree $T$ 
  to be a subtree $S$ containing the root.
  For varying $T$, these form a groupoid  $\treescuts$:
  its objects are the root
preserving inclusions $c:S\rightarrowtail T$, 
and its arrows are the
isomorphisms of such inclusions, i.e.~commutative diagrams

\begin{equation}\label{treescutsiso}
%\treescuts\quad=\quad\left(
\vcenter{\xymatrix{T\ar[r]^-\tau_-\cong&T'\\S\ar@{ >->}[u]^c
\ar[r]^-\sigma_-\cong&S'\ar@{ >->}[u]_{c'}.}}
%\right)
\end{equation}
This groupoid comes equipped with canonical morphisms $m,r:\treescuts\to\tr$ 
and $w:\treescuts\to\forest$:
when applied to a cut $c:S\rightarrowtail T$, the map
$m$ returns the total tree $T$, the map $r$
returns the subtree (i.e.~the tree $S_c$ found below the cut), and the map $w$ returns the 
forest $P_c$ consisting of the ideal trees in
$T$ generated by the leaves of $S$. These maps and the morphisms $L,R$ in \eqref{span} above form a commutative diagram
\begin{equation}\label{treecutmaps}
\vcenter{\xymatrix{
%
%\tr\ar[dd]_L&&\ar[ll]_m\treescuts\ar[rd]^(0.6)r\ar[ld]_(0.6){w\!\!\!}\\
%&\ar[ld]_L\forest\ar[rd]^R&&\ar[ld]_L\tr\ar[rd]^R\\
%{\fam I}&&{\fam I}&&I
%
%BETTER: TYPESET THIS AS A CUBE MINUS BOTTOM VERTEX
%
\treescuts\ar@{-->}[dr]|{\vphantom{|}m}
\ar[rr]^r\ar[dd]_w&&\tr\ar[dd]_(.35){^L}\ar@{-->}[dr]^{^R}
\\
& \tr \ar@{-->}'[r]_{^R}[rr]\ar@{-->}'[d]^{^L}[dd]&& I% \ar@{-}[dd]
\\
 \forest \ar[rr]^(.35){_R} \ar@{-->}[dr]_{^L}
   & & {\fam I}
\\  &{\fam I}
}}
\end{equation}
%Thus $\treescuts$ is a groupoid over $I$ and has two canonical maps to ${\fam I}$. 
%We have an equivalence of groupoids
%\begin{equation}\label{eqn:fub}
%\treescuts_{\leaves{n}}\simeq \grint{S}{{}_{\leaves{n}}}{\treescuts_S}
%\end{equation}
 We denote by ${}_T\treescuts$, $\treescuts_S$ and $\treescuts_N$ the fibres of the functors $m$, $r$ and $L\circ r$.
%,  and over $n\in{\fam I}$ and $\treescuts_S$ is the fibre of $r$ over a tree $S$.

 For a fixed tree $T$, the arrows of the groupoid ${}_T\treescuts$ are
$$
\vcenter{\xymatrix{T\ar[r]^-=&T\\R\ar@{ >->}[u]\ar[r]_-\cong&R'\ar@{ >->}[u]}}
$$
and since the vertical maps are monomorphisms,  we see that 
this groupoid has no nontrivial
automorphisms, and hence is equivalent to a discrete groupoid which we denote 
by $\cut(T)$; we refer to its objects as the cuts of $T$.
Thus
$
{}_T\treescuts\;\simeq\; \pi_0 ({}_T\treescuts)\;=\;\cut(T) 
$
% The set $\cut(T)$ clearly has an action of the group $\Aut T$.
and the double-counting lemma \ref{doublecounting} implies the following.
\end{blanko}

\begin{lemma}\label{tree:doublecounting}
We have equivalences of groupoids
\begin{align*}
  \grint{T}{\tr}{\cut(T)} \ \simeq \
  \grint{T}{\tr}{{}_T\treescuts}
% \sum_T{}_T\treescuts/\Aut T
&\;\simeq \;\treescuts
\;\simeq\;
\grint{S}{\tr}{\treescuts_S}
\;\simeq\;
\grint{N}{{\fam I}}{\treescuts_N}
% \sum_R\treescuts_R/\Aut R
\end{align*}
\end{lemma}
\medskip

% It is easy to identify the left-hand side of this equivalence:

The following Main Lemma states that the solid
square face of \eqref{treecutmaps} is a
(homotopy) pullback square and enables us to identify the fibres $\treescuts_S$
and $\treescuts_N$.

\begin{lemma}\label{tree:key}
The canonical morphism to the product
%\begin{equation}\label{eqn:c-over-ft}
$$
(w,r)\colon\treescuts\longrightarrow  \forest\times\tr
$$
%\end{equation}
that sends  $c:S\rightarrowtail T$
to $(P_c,S_c)$, induces an equivalence 
%over $\forest\times\tr$
$$
\treescuts\simeq  \forest\times_{\fam I}\tr.
$$
\end{lemma}
\begin{proof}
%The proof is formally similar to the main lemma for surjections, but relies on 
%a colimit construction in the category of forests rather than just a sum 
%construction as in the category of surjections.  
Starting with an
object $(P,\,S,\,L(S)\stackrel\lambda\cong R(P))$ of the pullback, we construct a tree with a cut by {\em grafting}.
The isomorphism $\lambda$ may be regarded as a root-preserving embedding of forests
$$LS\rightarrowtail P=\sum_{\ell\in LS}T_{\lambda(\ell)},$$
and we construct the pushout in the category of forests of this map and the 
ideal subforest embedding $LS \to S$,
$$\xymatrix{
\sum T_{\lambda(\ell)}\ar[r]&T\dlpullback\\LS\ar@{ >->}[u]\ar[r]&S\ar@{ >->}[u]
}$$
to obtain a root-preserving embedding $S\rightarrowtail T$  in the sense of \ref{forests}.
Note that since the forest $S$ is a tree, $T$ is again a tree.
This assignment is functorial:
an isomorphism $(\rho,\sigma)$ from $(P,S,\lambda)$ to $(P',S',\lambda)$ induces an isomorphism of pushouts $\tau:T\cong T'$ extending $\sigma$ as in \eqref{treescutsiso}.

In the reverse direction, we {\em prune} a root-preserving inclusion $S\rightarrowtail
T$ to obtain $(\sum T_\ell,S,\Id)$ where $T_\ell$ is the ideal subtree of $T$
generated by the image of the leaf edge $\ell$ in $T$.  An isomorphism of root-preserving inclusions \eqref{treescutsiso} is sent to $(\tau,\sigma)$ where
$\tau_\ell:T_\ell\to T_{\tau\ell}'$ is the restriction of $\tau$ to the ideal
subtree $T_\ell$.
\end{proof}

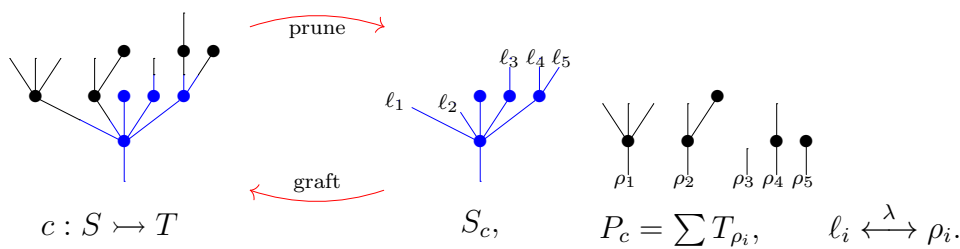
$$\xygraph{!{<0mm,0mm>;<3.9mm,0cm>:<0mm,6mm>::}
!~R{[]*-<2.5pt,2.5pt>[blue]{\bullet}} !~L{[]*-<.5pt,.5pt>{\phantom\bullet}} !~C{[]} 
!~N{[]*-<2.5pt,2.5pt>{\bullet}} 
%%%%
!{(-0.5,-0.8)} *{c:{%\color{blue}
S}\rightarrowtail T}
!{(12,-0.8)} *{{%\color{blue}
S_c,}}
!{(20,-0.8)} *{
\qquad\qquad P_c=\sum T_{\rho_i},
\qquad \ell_i\stackrel\lambda\longleftrightarrow\rho_i.
}
!{(4,3.5)}  !L="A"
!{(9,3.5)}  !L="B"
!{(4,0)}  !L="D"
!{(9,0)}  !L="C"
!{(0,0)}  !L="x"
!{(0,1)}  !R="c"
!{(0,2)}  !R="e"  !{(1,2)}  !R="f"  !{(2,2)}  !R="g"
!{(-3,2)} !N="h"  !{(-1,2)} !N="i"  !{(1,3)}  !L="j"  !{(2,3)}  !N="k"  !{(3,3)} !N="l"
!{(-4,3)} !L="m"  !{(-3,3)} !L="n"  !{(-2,3)} !L="o"  !{(-1,3)} !L="p"  !{(0,3)} !N="q"   !{(2,4)} !L="r"
!{(-1.5,1.5)} !C="ch" !{(-0.5,1.5)} !C="ci" !{(1,2.5)} !C="fj" !{(2,2.5)} !C="gk" !{(2.5,2.5)} !C="gl"
%%%%
!{(12,0)} !L="xx"
!{(12,1)} !R="cc"
!{(12,2)} !R="ee"  !{(13,2)} !R="ff"  !{(14,2)}  !R="gg"
!{(9,2)}  !L="hh"  !{(11,2)} !L="ii"  !{(13,3)}  !L="jj"  !{(14,3)} !L="kk" !{(15,3)} !L="ll"
%%%%
!{(17,0)} !L="td1" !{(19,0)} !L="td2"
!{(21,0)} !L="tf"  !{(22,0)} !L="tg1" !{(23,0)}  !L="tg2"
!{(17,1)} !N="th"  !{(19,1)} !N="ti"  !{(21,1)}  !L="tj"  !{(22,1)} !N="tk" !{(23,1)} !N="tl"
!{(16,2)} !L="tm"  !{(17,2)} !L="tn"  !{(18,2)}  !L="to"  !{(19,2)} !L="tp" !{(20,2)} !N="tq" !{(22,2)} !L="tr"
%%%%
"x"-@[blue]"c" "c"-@[blue]"ch"  "ch"-"h" "h"-"m" "h"-"n" "h"-"o" 
               "c"-@[blue]"ci"  "ci"-"i" "i"-"p" "i"-"q"
"c"-@[blue]"e" 
"c"-@[blue]"f" "f"-@[blue]"fj"  "fj"-"j"
"c"-@[blue]"g" "g"-@[blue]"gk"  "gk"-"k" "k"-"r" 
               "g"-@[blue]"gl"  "gl"-"l" 
%%%%
"xx"-@[blue]"cc" 
"cc"-@[blue]"hh"  |>{\ell_1\;}
"cc"-@[blue]"ii"  |>{\ell_2\;\;}
"cc"-@[blue]"ee" 
"cc"-@[blue]"ff" "ff"-@[blue]"jj" |>{\ell_3}
"cc"-@[blue]"gg" "gg"-@[blue]"kk" |>{\ell_4\;} 
                 "gg"-@[blue]"ll" |>{\ell_5\;}
%%%%
"td1"-"th" |<{\rho_1\;} "th"-"tm" "th"-"tn" "th"-"to" 
"td2"-"ti" |<{\rho_2\;} "ti"-"tp" "ti"-"tq"
"tf"-"tj"  |<{\rho_3\;}
"tg1"-"tk" |<{\rho_4\;} "tk"-"tr" 
"tg2"-"tl" |<{\rho_5\;}
%%%%
"A":@/^1ex/@[red]"B"_{\text{prune}}
"C":@/^1ex/@[red]"D"_{\text{graft}}
}
$$

\begin{cor}
For $S\in \tr$ and $N\in{\fam I}$ we have equivalences of groupoids
\begin{align}
\treescuts_S&\simeq  (\forest\times_{\fam I}\tr)_S\simeq \forest_{LS}
,\nonumber%\label{}
\\
\treescuts_N&\simeq \forest_N\times {}_N\tr 
.\nonumber%\label{}
\end{align}
\end{cor}

Combining the previous results, we arrive at our main theorem:
\begin{thm}\label{main-thm-grpd}
  We have equivalences of groupoids %over $\forest\times\tr$:
\begin{align*}
\grint{T}{\tr}{\cut(T)}& 
\simeq \grint{S}{\tr}{\forest_{LS}} 
\\&
\simeq \grint{\leaves{N}}{{\fam I}}{\forest_N \times {}_{\leaves{N}}\tr} 
.
\end{align*}
\end{thm}
We can regard this as an equivalence of groupoids over $\forest\times\tr$.  For fixed $T$, the map from ${\cut(T)}$ to $\forest\times\tr$
is precisely
$$
\sum_{c\in \cut(T)}\xymatrix{1  \ar[rr]^-{\name{(P_c,S_c)}}&&\forest\times\tr} .
$$
To emphasise this, we can reformulate the result as
\begin{equation}\label{fdbgrpd}
\grint{T}{\tr}\sum_{c\in \cut(T)} \{P_c\} \times \{S_c\}
\simeq \grint{\leaves{N}}{{\fam I}}{\forest_N\times {}_{\leaves{N}}\tr}
\end{equation}

% \begin{blanko}{$B_+$-remark.}
%   The key lemma gives $\treescuts \simeq \forest \times_{\fam I} \tr$.  We also have the
%   projection $m:\treescuts \to \tr$ which returns the total tree.  The composite
%   $$
%   \forest \times_{\fam I} \tr \simeq \treescuts \stackrel m \to \tr
%   $$
%   deserves perhaps to be called $B_+$: it takes a tree $S$ and an $LS$-indexed
%   family of trees, and graft all those trees onto the leaves of $S$, and returns
%   the total tree.  When $S$ is an elementary tree (i.e.~a tree with just one
%   node), this is an analogue of the Hochschild $1$-cocycle which can be used to
%   characterise the Connes--Kreimer Hopf algebra, and which features in the
%   combinatorial Dyson--Schwinger equation, the starting point for other
%   relations with the \fdb formula,
%   cf.~e.g.~Foissy~

%   The comultiplication $\Delta$ will be something like $B_+^{-1}$ (taking fibre),
%   except for the right-hand factor.
% \end{blanko}

Extracting the algebraic version of the \fdb formula \ref{main-thm-alg}
from \ref{main-thm-grpd} will be a matter of taking cardinality in a certain sense,
which we explain in the next section.

If we take the fibres of the equivalence
given in Theorem \ref{main-thm-grpd},
over a fixed colour $v\in I$,
we obtain:
\begin{cor}\label{main-cor-grpd}
  We have equivalences of groupoids %over $\forest\times\tr$:
\begin{align*}
\grint{T}{\tr_v}{\cut(T)}& 
%\simeq \grint{S}{\tr}{\forest_{LS}} 
%\\&
\simeq \grint{\leaves{N}}{{\fam I}}
{\forest_N \times {}_{\leaves{N}}\tr_v} 
.
\end{align*}
\end{cor}

\section{Groupoid cardinality}
\label{Sec:card}

\begin{blanko}{Finiteness conditions and cardinality.}\label{finite}
  A groupoid $X$ is called {\em finite} %(sometimes also called {\em finite})
  when $\pi_0 (X)$ is a finite set and each $\pi_1(x)$ is a finite group.  
A morphism of groupoids is
  called finite when all its fibres are finite.
%\end{blanko}

%\begin{blanko}{Cardinality.} 
  The {\em cardinality} \cite{Baez-Dolan:finset-feynman} of a finite groupoid (sometimes called {\em groupoid
  cardinality} or {\em homotopy cardinality} if there is any danger of confusion) is the nonnegative
  rational number given by the formula
  $$
  \norm{X} := \sum_{x\in \pi_0 X} \frac{1}{\norm{\Aut(x)}} .
  $$
  Here $\norm{\Aut(x)}$ denotes the order of the vertex group at $x$.  This is
  independent of the choice of the $x$ in the same connected component since an
  arrow between two choices induces an isomorphism of vertex groups.  
%The
%  cardinality of a groupoid coincides with that of any skeleton, so the
%  following fundamental result is clear from Lemma~\ref{skel-lemma}:
%
%\begin{lemma}
It is clear that equivalent groupoids have the same cardinality.
%\end{lemma}

If $X$ is a finite set
  considered as a groupoid, then the groupoid cardinality coincides with the set
  cardinality.  If $G$ is a finite group considered as a one-object groupoid, then the
  groupoid cardinality is the inverse of the order of the group. 
% The groupoid
%  cardinality is a standard construction in physics and combinatorics: you sum
%  over the different (isomorphism classes of) objects and for each object divide out by the order of its
%  symmetry group.

% The {\em sum} of two groupoids is just their disjoint union.  The {\em product}
% is defined by taking the product of the sets of objects and the product of the
% set of arrows.  
We have the following fundamental formulae for cardinality of sums, products  and homotopy quotients 
of groupoids:
\begin{align*}
  \norm{X+Y} & = \norm X + \norm Y\\
  \norm{X\times Y} & = \norm X \times \norm Y\\
  \norm{X/G} &= \norm{X}/\norm{G},
\end{align*}
where $X$ and $Y$ are finite groupoids and $G$ is a finite group acting on $X$.
%extending the analogous results for the cardinality of finite sets.
\end{blanko}

%The following is one important feature of 
%homotopy quotients and cardinality.

%\begin{lemma}\label{actioncardinality}
%  For any action of a finite group $G$ on a finite groupoid $X$, we have
%  $$%
%
%  $$
%  where
%  $\norm G$ denotes the order of the group $G$.
%\end{lemma}

 \begin{blanko}{Cardinalities of families.}\label{card-fam}
  For the sake of taking cardinalities we shall need the following
`numerical' description of the groupoid $\fam I$ of families of objects in $I$, cf.\ \ref{fam}.  

Let $v_1,\ldots,v_s$ be representatives of the isoclasses in $I$. Then
every family
$$N:X\to I$$ is isomorphic to a
sum (in the category of sets over $I$) of families of the kind $\name{v_i}:1 \to I$.
Hence for uniquely determined natural numbers $n_i$
%and
%let $F_i$ denote the family 
%$$ 1 \to \{v_i\}/\Aut(v_i)$$ 
%and let $n_iF_i$
%denote the \fbox{sum of $n_i$ such families, i.e.~the} family $n_i \to \{v_i\}/\Aut(v_i)$.
%\marginpar{\fbox{\begin{minipage}{7em}\tiny In these 3 sentences the word `sum' has 2 different meanings?\end{minipage}}} 
we have
$$N \cong \sum_{i=1}^s n_i \name{v_i} .$$
It follows that $$\pi_0(\fam I) \simeq \N^s.$$

We compute the vertex group.  The automorphism group of
$\name{v_i}:1\to I$ is $\Aut(v_i)$ and that of $n_i\name{v_i}$ is
$n_i! \Aut(v_i)^{n_i}$,
since each point contributes with a factor $\Aut(v_i)$, and since the
points can also be permuted.
Altogether, we have
\begin{equation}\label{autN}
\Aut(N)\cong\prod_{i=1}^s{n_i! \Aut(v_i)^{n_i}},
\end{equation}
and the groupoid $\fam I$ can be described as
$$
\fam I \simeq \sum_{(n_1,\ldots,n_s)\in \N^s} 
\;\frac{1}{\displaystyle \prod_i{n_i! \Aut(v_i)^{n_i}}} .
$$

% REMARK,  NOT NEEDED.  If $I$ is a one-object groupoid with vertex
% group $G$, then
% $\fam I$ seems to be the category of finite free $G$-sets.
% Indeed, the groupoid valued functor on $I$ that corresponds to the family
% $X \to I$ is the presheaf
% $$
% \bullet \mapsto X_\bullet = X \times \norm G
% $$
% which is the free $G$-set on $X$.
\end{blanko}

\begin{blanko}{Relative cardinality.}\label{formalcard}
%  Groupoid cardinality makes sense for more general groupoids than the finite 
%  ones: for example, the groupoid of finite sets and bijections has cardinality 
%  $\sum_{n\geq 0} \frac1{n!} =e$, see Baez--Dolan~\cite{Baez-Dolan:finset-feynman}.
% 
%  We shall instead make use of 
Relative cardinality
%  which 
refers to the situation where one
  groupoid $X$ is relatively finite over another 
  groupoid,   i.e.~we have a morphism $p:X\to B$ with finite fibres.
  This notion is from \cite{Baez-Hoffnung-Walker:0908.4305}.
  In   this situation we define the {\em relative cardinality} of $X$ relative to 
  $B$,
  $$
  \norm{p} := \norm{X}_B := \sum_{b\in\pi_{0}B} \frac{\norm{X_{b}}}{\norm{\Aut(b)}}
  \cdot \delta_{b} ,
  $$ 
in the completion of the $\Q$-vector space 
  spanned by symbols $\delta_{b}$ for $b\in\pi_{0}(B)$.
The notation $\norm{X}_B$ assumes the
 morphism $X \to B$ is clear from the 
context.
  Since the morphism has finite fibres $X_b$, the 
  coefficients are well-defined nonnegative rational numbers.

  The vector space spanned by the $\delta_b$ is isomorphic to the space of functions 
  $\pi_0 B \to\Q$ with finite support, and its completion is the function space
  $\Q^{\pi_0 B }$. For each $b\in \pi_0 B$ we identify the cardinality of the inclusion $\name b:1 \to B$ with a function
\begin{eqnarray*}
    \delta_b=\norm {1}_{\name b}
\ : \ \pi_0 B & \longrightarrow & \Q \\
    x & \longmapsto & \begin{cases} 1 & \text{ if } x \simeq b \\ 0 & 
    \text{ otherwise. }\end{cases}
\end{eqnarray*}
Hence we identify the relative cardinality of $X\to B$ with the function 
\begin{eqnarray*}
%    |X|_{B}=\sum_{b\in\pi_{0}B} \norm{X_{b}}/\norm{\Aut(b)} 
%  \cdot \delta_{b} 
%\ : \ 
\pi_0 B  & \longrightarrow & \Q \\
    b & \longmapsto & \norm{X_{b}}/\norm{\Aut(b)} .
\end{eqnarray*}
% Thus the relative cardinality of $p:X\to B$ is obtained by composing 
% absolute groupoid cardinality with the corresponding functor $p':\pi_0 B\to\Grpd$ we saw in \eqref{Groth}.
\end{blanko}

\begin{blanko}{Properties of relative cardinality.}
    When taking relative cardinality of a product 
    $p\times p'\colon X\times X'\to B\times B'$, the formal symbols are indexed by $(b,b')\in 
    \pi_0 B \times \pi_0 B'  \simeq \pi_0(B\times B')$.  We shall then
    use notation $\delta_{b} \otimes \delta_b'$ instead of 
    $\delta_{(b,b')}$, so  that $|p\times p'|=|p|\otimes |p'|$. 
% We shall need the following obvious 
%    compatibility with products:  if 
%    $$
%    \xymatrix{X\times X' \ar@{}[rd]|{=}\ar[d]& X\ar[d]  \ar@{}[rd]|{\times}
%     & X'\ar[d]
%    \\ B\times B' & B & B'  }
%    $$
%    then 
%    $$ \norm{X\times X'}_{B\times B'} = \norm{X}_B \otimes 
%    \norm{X'}_{B'} .
%    $$
\end{blanko}

Consider the groupoid morphism $X\to X/G
$ given by  the action of a finite group on a groupoid. Then we have
$$
\norm{X/G}_{X/G}=\frac{\norm{X}_{X/G}}{\norm{G}}
$$
%and we have the following generalisation of Lemma \ref{actioncardinality}:
\begin{lemma}\label{actionformalcardinality}
  For any action of a finite group $G$ on a groupoid $X$ and a finite morphism $X/G\to A$, we have
  $$
  \norm{X/G}_A = \norm{X}_A/\norm{G}
  $$
  where
  $\norm G$ denotes the order of the group $G$.
\end{lemma}

We need the following transitivity property of relative cardinality:
\begin{lemma}\label{relrel}
  Given groupoid morphisms $X \stackrel p\to B \stackrel t \to I$ with finite fibres,
the relative cardinality of $p$ is obtained from those of the restrictions $p_v:X_v\to B_v$,
$$
\norm{p} = \norm{X}_B = \sum_{v\in \pi_0 I} \frac{\norm{p_v}}{\norm{\Aut(v)}} .
$$  
Also the relative
  cardinality of $X$ over $I$ is obtained from the relative cardinality over
  $B$ by substituting $\delta_{t(b)}$ for each $\delta_b$.  That is:
$$
\norm{X}_I = \sum_{b\in \pi_0 B} \frac{\norm{X_b}}{\norm{\Aut(b)}} \; \delta_{t(b)} .
$$  
\end{lemma}

In particular, any groupoid can be measured over itself via the identity
morphism $\Id: X \to X$:
$$
\norm{X}_X = \sum_{x\in \pi_0 X} \frac{1}{\norm{\Aut(x)}} \; \delta_x
$$
Hence we get the following useful result.
\begin{cor}\label{formalcardcor}
  For $p:X\to B$ we have
$$
\norm{X}_B = \sum_{x\in \pi_0 X} \frac{1}{\norm{\Aut(x)}} \; \delta_{p(x)} .
$$  
\end{cor}

\begin{blanko}{Power series.}\label{Gascard}
  Both the power series $A$ from Section~\ref{sec:clas} and the Green function in \ref{tree:green}
%of $P$-trees defined in the next section 
are examples of the following general situation.
% There is a general notion of Green function which includes
% the power series $A$ in the \fdb algebra $\fdbsymb$ and also
% our Green functions of $P$-trees. 
For $X$ a groupoid with
finite vertex groups, 
% and consider first the relative cardinality
% of %$X$ relative to itself, with respect to
% the identity functor $X\to X$,
% $$\sum_{x\in\pi_0 X}\frac1{\norm{\Aut(x)}}\cdot \delta_x .$$
% Equally natural is to
consider the relative cardinality of the map 
$X\into \fam X$, the full inclusion
of $X$ into the groupoid of families of objects in $X$ (cf.~\ref{fam}).
We have
% We define the
%  {\em Green function} as the relative cardinality of this inclusion
$$\sum_{x\in\pi_0 \widetilde X}\frac{\norm{X_x}}{\norm{\Aut(x)}}\cdot \delta_x
=
\sum_{x\in\pi_0 X}\frac1{\norm{\Aut(x)}}\cdot \delta_x ,
$$
since the summand is zero when $x\notin\pi_0 X$.
As an element of the ring
$\Q^{\pi_0 \widetilde X}$, this cardinality is
\begin{eqnarray*}
  \pi_0\widetilde X & \longrightarrow & \Q \\
    x & \longmapsto & 
\begin{cases} 1/{\norm{\Aut(x)}} & \text{ if } x \in \pi_0 X \\
 0 &    \text{ otherwise. }\end{cases}
\end{eqnarray*}

We observe
the natural isomorphisms
$$
\Q^{\pi_0 \widetilde X} \cong \text{Sym} (\Q^{\pi_0 X})\cong 
\Q[[\delta_x%,\;x\in\pi_0 X
]]_{x\in\pi_0 X}
$$
between the ring of functions and the power series in symbols $\delta_x$ 
for $x\in\pi_0 X$. This restricts to an isomorphism between the subring of functions with finite support and the polynomial ring.

% Important examples: if $X$ is the groupoid of nonempty finite sets,
% then $\widetilde X$ is the groupoid of surjections (partitions).
% In this case, the Green function is the $A$.

% If $X$ is the groupoid of trees (or $P$-trees), then $\widetilde X$ is 
% the groupoid of forests (or $P$-forests).

% If $X$ is the trivial (contractible) groupoid, then $\widetilde X$ is the 
% groupoid of finite sets and bijections. The Green function is
% $$G = x 
% $$
% in $\Q[[x]]$.
\end{blanko}

\section{The \fdb formula in the bialgebra of trees}
\label{fdbalg}

The polynomial functor $P$, represented by  $
  I \lTo E \rTo B \rTo I
  $, remains fixed in this section, and we denote the groupoids of $P$-trees and $P$-forests by $\tr$ and $\forest$ respectively.

\begin{blanko}{The bialgebra of $P$-trees  (cf.~\cite{Kock:1109.5785}).}
Consider the bialgebra 
given by the free algebra on the set 
of isomorphism classes of $P$-trees
$$\treebialg=\Q[\delta_T]_{T\in\pi_0\tr}.$$
The constituent trees of a $P$-forest $F$ define a monomial $\delta_F$, and these monomials form a linear basis of $\treebialg$. The  comultiplication structure is given by
$$\Delta(\delta_T)=\sum_{c\in \cut(T)} \delta_{P_c} \otimes \delta_{S_c}.$$
The bialgebra of~\ref{treebialg} is just the special case where $P$ is the exponential functor of~\ref{naked}.
\end{blanko}

\begin{blanko}{Definition of the Green functions of trees.}
%  We first give the `numerical' definition of Green functions of various 
%  flavours and indicate the relationship between them, and state the \fdb
%  formulae.  Afterwards we interpret 
%  each of the Green functions in terms of groupoid cardinality, and from 
%  these interpretations
%  the relations become obvious, and finally we derive the \fdb formulae.
In the completion of $\treebialg$,  %given by 
the power series ring, we define the total Green function %in the bialgebra of $P$-trees
  as the relative cardinality of $\tr \to 
  \fam\tr = \forest$:
  $$
  G := \sum_{T \in \pi_0 \tr} 
  \frac{\delta_T}{\norm{\Aut(T)}}\;\;\in\;\;\Q[[\delta_T]]_{T\in\pi_0\tr} .
  $$
% By the above discussion this is the relative cardinality of the identity functor on $\tr$ or of the inclusion of $\tr$ in $\forest$. 

In analogy with the situation in QFT, where there is
  one Green function for each possible residue (interaction label) in the 
  theory, we also define an individual Green function for each possible 
(isomorphism class of)
root colour
   $v\in \pi_0 I$,
   $$
   G_v := \sum_{T \in \pi_0(\tr_v)} \frac{\delta_T}{\norm{\Aut_v(T)}} .
   $$
   Here the automorphism group $\Aut_v(T)$ consists of those automorphisms
   of $T$ which fix the root colour $v$.  
This is the relative cardinality of the inclusion $\tr_v \to\tr \to \forest$.

It follows from Lemma \ref{relrel} that we have the relationship
$$
G = \sum_{v\in \pi_0 I} \frac{G_v}{\norm{\Aut(v)}}.
$$

Let $s := \norm{\pi_0 I}$ be the number of colours, and
let $\fat n = (n_1,\ldots,n_s) \in \N^s$ be a multiindex, parametrising an 
isoclass of objects $N$ in $\fam I$.  
Consider %$N\in \fam I$ and 
the  relative cardinality of the inclusion of the homotopy fibre
% ${}_N\!\tr$ 
%$$
%\xymatrix{
%{}_N\tr\ar[r]\ar[d]\drpullback & \tr \ar[d]^L \\
%{\mathbf 1}\ar[r]_{\name{N}}&\fam I
%}
%$$
%and the relative cardinality of 
${}_N\!\tr
\to \tr$,
$$G_{\fat n}:=\sum_{T \in \pi_0({}_N \!\tr)}\frac{\delta_T}{|\Aut_N(T)|}.$$
We also consider the summands of
 the Green
function corresponding to all trees with $n_v$ leaves of each colour $v\in \pi_0 I$,
$$g_{\fat n}:=\sum_{T \in \pi_0 \tr,\;LT\cong N}\frac{\delta_T}{|\Aut T|}.$$
This is the relative cardinality of 
the {\em full subcategory} of $\tr$ whose objects are those trees $T$ with leaf profile $N$.
This  is equivalent to the homotopy quotient ${}_N\!\tr/\Aut N$ of the
homotopy fibre by the canonical action of $\Aut N$. Clearly, 
%
%subcategory is the full fibre of $L:\tr\to\fam I$ over $N$ and
%It follows from Lemma \ref{actionformalcardinality} that
\begin{eqnarray*}
g_{\fat n}&=&G_{\fat n}/|\Aut N|
%\\&=&\sum_{T \in ({}_N\!\tr/\Aut N)_0}\frac{\delta_T}{\Aut T}
%\\&=&\sum_{T\in ({}_N\! \tr)_0}\frac{\delta_T}{|\Aut_N T||\Aut N|}
.\end{eqnarray*}
and hence
$$
G = \sum_{\fat n \in \N^s} g_{\fat n} .
$$
%In fact the Green function $g_{\fat n}$ is the relative cardinality of the  

%\footnotesize

%\normalsize

\end{blanko}

The comultiplication $\Delta$ extends to power series, and we can extract an 
algebraic \fdb formula from of our Theorem \ref{main-thm-grpd}.

\begin{thm}\label{main-thm-alg}
  The following \fdb formula holds for the Green function in the bialgebra of 
  trees.
\begin{equation}\label{fdbcard}
\Delta(G)=\sum_{\fat n \in \N^s} G^{\fat n}\otimes g_{\fat n}.
\end{equation}
Here $G^{\fat n}$ is to be interpreted as the product
$$
G^{\fat n} = \prod_{v\in \pi_0 I} G_v ^{n_v} .
$$
\end{thm}

%%%%%%%%%%%%%%%%%%%%%%%%%%%%%%%%%%%%%%%%%%%%%%%%%%
%\begin{blanko}{Groupoid interpretation of the Green functions.}
%%%%%%%%%%%%%%%%%%%%%%%%%%%%%%%%%%%%%%%%%%%%%%%%%%
To prove the theorem we first need a result about forests.
Recall that the multiindices $\fat n$ classify the isomorphism classes of objects
$N \in \fam I$.

\begin{lemma}\label{famprod}
Let $N:X\to I$ be an object of $\fam I$ of class $\fat n=(n_1,\dots, n_s)$.
Then 
$$
\forest_N\simeq\prod_{i=1}^s{\tr_{v_i}}^{n_i}.
$$
\end{lemma}
\begin{proof}
Combining  Lemma~\ref{Fn=T^n} with \ref{card-fam}, we find
\begin{eqnarray*}
\forest_N&\simeq&
\Grpd_{/I}(X,R)
\\
%&\simeq&
%\Grpd_{/I}\left(\grint v I {E_v},\tr\right)
%\\
&\simeq&
\Grpd_{/I}\left(\sum_{i=1}^s {n_i\;\name{v_i}}
,R\right)
\\
&\simeq&
\prod_{i=1}^s\Grpd_{/I}\left(
 {\name{v_i}}
,R\right)^{n_i}.
\end{eqnarray*}
Now $\name{v_i}$ is the `lowershriek' $\name{v_i}\lowershriek(1)$ and so 
by adjunction~\eqref{slice-adj} we have
$$
\simeq\;\;
\prod_{i=1}^s\Grpd (
1
,\name{v_i}\upperstar R )^{n_i}
\;\;\simeq\;\;\prod_{i=1}^s{\tr_{v_i}}^{n_i}.
$$
%Although $I$ may not be a discrete groupoid, the full fibres 
%$$
%{E_v}/\Aut v\longrightarrow \sum_{v\in I_0} \{v\}/\Aut v\stackrel\sim\longrightarrow I
%$$ are:  they are subsets-over-$I$ of $N:E\to I$, of cardinality $n_v$ say, and
%\marginpar{adjunction argument with lowershriek of $\name{v_i}$???}
%$$
%\forest_N\;\;\simeq\;\;\Grpd\left(\sum_{v\in I_0} n_v,\tr_v\right)\;\;\simeq\;\;\prod_{v\in I_0}{\tr_v}^{n_v}.
%$$
\end{proof}
\begin{cor}
$$
\norm{\forest_N}=G^{\fat n}=\prod_{i=1}^s{G_{v_i}}^{n_i}.
$$
\end{cor}

% \begin{blanko}{Groupoids for individual Green functions.}
% % For each
% % edge colour $v\in I$, there is a Green function $G_v$, and we have
% % $$
% % G = \sum_{v\in\pi_0 I} G_v/\norm{\Aut(v)}.
% % $$
% % 
%   Let $\tr_v$ denote the groupoid of trees with root colour $v\in I$, and let
%   ${}_\leaves{N}\tr_{v}$ denote the groupoid of trees with leaf profile 
%   $\leaves{N}\in \fam I{}_0$ and root colour $v$.  Then we also get
%  $$
% \grint{T}{\tr_v}\sum_{c\in \cut(T)} \{P_c\} \times \{S_c\}
% \simeq \grint{\leaves{n}}{\fam I}{\forest_{\leaves{n}}\times 
% {}_\leaves{n}\tr_{v}}
% $$
% and on taking groupoid cardinality
% $$
% \Delta(G_v)=\sum_nG^{n}\otimes g_{n,v}.
% $$
% %This is more in the style of van Suijlekom\ldots
% \end{blanko}

\begin{blanko}{Proof of theorem \ref{main-thm-alg}.}
The left-hand side $\Delta(G)$ of \eqref{fdbcard} is the relative cardinality 
of the left-hand side of
\eqref{fdbgrpd} of theorem \ref{main-thm-grpd}.
It remains to show that the right-hand side of 
\eqref{fdbcard} is the relative cardinality of the right-hand side of
\eqref{fdbgrpd}. We have
%\begin{eqnarray*}
$$  
\norm{\grint{N}{{\fam I}}{\forest_N\times {}_N\tr}} 
=  \sum_{N\in\pi_0 \fam I}\norm{\forest_N}
\otimes \norm{{}_N\tr}/\norm{\Aut N}  
   = \sum_{\fat n\in \N^s} G^{\fat n} \otimes g_{\fat n} .
$$%\end{eqnarray*}
% in which $\norm{{}_N\tr}/\norm{\Aut N}$ is just $g_N$.
% Therefore the relative cardinality of \eqref{fdbgrpd}, over $\forest\times\tr$,  gives the formula
% $$
% \Delta(G)\;\;=\;\;\sum_{N\in \pi_0 \fam I}\prod_{v\in \pi_0 I}{G_v}^{n_v}\otimes g_N
% $$
% as required.
\qed
\end{blanko}

\begin{blanko}{Summands of Green functions.}
If $v\in I$ and $\fat n \in \N^s$ 
is a multi\-index parametrising an 
isoclass of an object $N\in\fam I$, define the Green
function
% of all trees with $n_v$ leaves of each colour $v\in I_0$ as 
$$
g_{\fat n,v} := \norm{ {}_N\tr_v/\Aut N}.
$$
We have
$$
g_{\fat n} = \sum_{v\in \pi_0 I}  \frac{g_{\fat n,v}}{\norm{\Aut(v)}}
$$
and hence
$$
G = \sum_{\fat n} \sum_{v\in \pi_0 I} \frac{g_{\fat n,v}}{\norm{\Aut(v)}} .
$$
\end{blanko}

Taking relative cardinality of Corollary \ref{main-cor-grpd} 
then gives
\begin{thm}
For $v\in I$ and $\fat n \in \N^s$ we have
$$
\Delta(G_v)=\sum_{\fat n \in \N^s} G^{\fat n}\otimes g_{\fat n,v}.
$$
\end{thm}
This is the version that most closely resembles the multi-variate \fdb formula
and the formula of van Suijlekom.

\section{Examples}\label{sec:ex}
%%%%%%%%%%%%%%%%%%%%%%%%%%%%%%%%%%%%%%%%%%%%%%%%%%

In this section we specialise to some standard examples of the polynomial
endofunctor $P$, and compare with the classical \fdb bialgebra.  Following
\cite{Kock:graphs-and-trees} we also explain a polynomial endofunctor of certain
graphs which was actually our motivating example, and which points towards
transferring our results to bialgebras of graphs.

\begin{blanko}{Naked trees.}\label{naked}
  Consider the polynomial functor $P$ represented by
$$
1\leftarrow \Bij'\to \Bij\to 1,
$$
where $\Bij'$ denotes the groupoid of finite pointed sets and
basepoint-preserving bijections, and $1$ denotes a singleton set.
This is the exponential
functor $$P(X) = \exp(X) = \sum_{n\in \N} X^n / n!.$$
There is a fibre of each finite cardinality $n\in \N$, and for every
tree $A \leftarrow M \to N \to A$ there is a unique $P$-decoration
$$    \xymatrix{
  A \ar[d]& \ar[l]  M\ar[d] \drpullback \ar[r] & N\ar[d] \ar[r] 
& A\ar[d] \\
  1  &\ar[l] \Bij'\ar[r] & \Bij \ar[r]  &1 
  }$$
(since a node of arity $n$ must map to $n\in \Bij$, and since the
choices of where to map the incoming edges to the fibre over $n$
are all uniquely isomorphic).  It follows that in this case $P$-trees are 
essentially the same thing as the naked trees defined in \ref{polytree-def}
(in the precise sense that the groupoid of $P$-trees is equivalent to the 
groupoid of naked trees).
\end{blanko}

\begin{blanko}{Cyclic trees.}\label{example:cyclic}
  If $P$ is the polynomial endofunctor
$$
1\leftarrow \C'\to \C\to 1,
$$
where $\C$ is the groupoid of finite cyclically ordered sets, and
$\C'$ is the groupoid of finite cyclically ordered pointed sets (in fact,
canonically equivalent to the $\N'$ of the following example), then
the notion of $P$-tree is that of cyclic tree.
\end{blanko}

\begin{blanko}{Planar trees.}
  Consider the polynomial functor $P$ represented by
$$
1\leftarrow \N'\to \N\to 1,
$$
where $\N$ is the (discrete) groupoid of finite ordered sets, and
$\N'$ is the (discrete) groupoid of finite ordered sets with a marked point,
so that the fibre of the middle map is naturally a linearly ordered set.
This functor is the geometric series $$P(X)= \frac{1}{1-X} = \sum_{n\in \N} 
X^n.$$
In this case the $P$-trees
$$    \xymatrix{
  A \ar[d]& \ar[l]  M\ar[d] \drpullback \ar[r] & N\ar[d] \ar[r] 
& A\ar[d] \\
  1  &\ar[l] \N'\ar[r] & \N \ar[r]  &1 
  }$$
are naturally planar trees, since the cartesian square in the middle 
equips the incoming edges of each node in the tree with a linear order.

Note that the resulting bialgebra of planar trees is still commutative, unlike the
planar-tree Hopf algebra studied by Foissy~\cite{Foissy:MR2409411}
%.
%~\cite{Foissy:2002I} 
and others.  Since $P$-trees
are rigid (this is true in general when $P$ is represented by discrete 
groupoids), there are no symmetries, so the Green function is just the
sum of all the formal symbols, 
$$G = \sum_{T \in \pi_0\tr} \delta_T .
$$
\end{blanko}

\begin{blanko}{Planar binary trees.}
Consider now the diagram
$$
1\leftarrow 2\to 1\to 1
$$
representing the polynomial functor $P(X)=X^2$.  In this case
$P$-trees are planar binary trees.
\end{blanko}

\begin{blanko}{Injections.}
For the constant polynomial functor $P(X)=1$, represented by
$$
1\leftarrow 0\to 1\to 1,
$$
there are two possible $P$-trees:
$$
x\inlineDotlessTree \qquad y \inlineNullaryTree
$$
$P$-forests are disjoint unions of these.
The groupoid $\forest$ of $P$-forests is naturally equivalent to the 
groupoid whose objects are injections between finite sets, and whose 
arrows are the isomorphisms between such.
The associated \fdb bialgebra is $\Q[\delta_x,\delta_y]$, with the comultiplication given by
\begin{align*}
\Delta(\delta_x) &= \delta_x\otimes \delta_x \\
\Delta(\delta_y) &= 1 \otimes \delta_y + \delta_y\otimes \delta_x .
\end{align*}
Expanding we find
$$
\Delta(\delta_y^n) = \sum_{k\leq n} {n \choose k} \delta_y^k \otimes \delta_y
^{n-k} \delta_x^k .
$$
After passing to the reduction (putting $x=1$) we get the usual
binomial Hopf algebra. The Green function is 
$$
G = \delta_x+\delta_y ,
$$
with $g_0 = \delta_y$ and $g_1=\delta_x$, and the \fdb formula is immediate.
\end{blanko}

\begin{blanko}{Linear trees.}
  The identity functor $P(X)=X$ is represented by
$$
1\leftarrow 1\to 1\to 1.
$$
Now $P$-trees are linear trees.  We take a variable $x_n$ 
for the isoclass of the linear tree with $n$ nodes, and find the
comultiplication formula
$$
\Delta(x_n)=\sum_{i=0}^n x_i\otimes x_{n-i};
$$
this is the ladder Hopf algebra, studied for example in \cite{mekr:2004}.
\end{blanko}

\begin{blanko}{Trivial trees.}
   Consider the polynomial functor
$$
P=(I\leftarrow 0\to 0\to I).
$$
where $I$ is a discrete groupoid.
The only $P$-trees are the trivial trees, one for each $x\in \pi_0 I$.
The groupoids of $P$-trees 
and $P$-forests are $I$ and $\fam I$ respectively.
In $\Q[\pi_0 I]$ all generators are grouplike, and we have 
\begin{align*}
G&=\sum_{x\in \pi_0 I}x
\\
\Delta(G)&=\sum_{x\in \pi_0 I}x\otimes x\;\;=\;\;\sum_{x\in \pi_0 I}| I_x\times{}_xI|
\;\;=\;\;\sum_{x\in\pi_0 \fam I}|\fam I_x\times{}_xI|
\end{align*}
(This is the monoid algebra on the free commutative monoid on $\pi_0 I$.)
\end{blanko}

\begin{blanko}{Effective trees.}
Consider the polynomial functor represented by
$$
1\leftarrow \nonempt'\to \nonempt\to 1,
$$
where $\nonempt$ is the groupoid of non-empty finite sets and bijections
(and $\nonempt'$ the groupoid of non-empty finite pointed sets and 
basepoint-preserving bijections).  The resulting endofunctor is
$P(X)=\exp(X)-1$.  In this case $P$-trees are naked trees with no nullary 
operations, sometimes
called `effective' trees.  These are the key to understanding the relationship 
with the classical \fdb bialgebra, cf.~\ref{fdb}, as explained below.

Since effective trees have no nullary nodes, they always have a non-empty set of
leaves, and therefore the leaf map can be seen to take values in $\nonempt$.
Furthermore, for each $n\in \nonempt$, the homotopy fibre ${}_n\tr \subset \tr$
is discrete, since if an automorphism of a effective tree fixes the leaves then
it fixes the whole tree.

The sub-bialgebra $\efftreebialg$ of $\treebialg$ is the polynomial algebra on 
the isomorphism classes of effective trees.
\end{blanko}

\begin{blanko}{Stable trees.}
  In a similar vein, we can consider $P$-trees for the polynomial
  functor $P(X)=\exp(X)-1-X$, represented by
$$
1\leftarrow {\bf Y}'\to {\bf Y}\to 1,
$$
where $\bf Y$ is the groupoid of finite sets of cardinality at least $2$.
These are naked trees with no nullary and no unary nodes, called {\em reduced 
trees} by
Ginzburg and Kapranov~\cite{Ginzburg-Kapranov}.  We adopt instead the term {\em stable
trees}.  
Clearly stable trees are effective, so $L: \tr\to\nonempt$ is a discrete
fibration. % REMOVE?
In this case it is furthermore finite: for a given number of leaves
there is only a finite number of isoclasses of stable trees. 
% This finiteness is convenient for computational purposes, and we include 
% an instructive computation in the appendix for this case.
% 
% Now the homotopy fibres ${}_n\tr$ are finite discrete
% subgroupoids of the groupoid of $P$-trees.
% 
% The sub-bialgebra $\redtreebialg$ of $\treebialg$ is 
% the polynomial algebra on the isomorphism classes of reduced trees.
\end{blanko}

%\begin{blanko}{Change of decoration and structure.}
%  From a cartesian natural transformation $P\Rightarrow Q$ we get a morphism of
%  groupoids $\tr_P \to \tr_Q$ just by change of decoration.  Since the bialgebra
%  structure only concerns the tree structure, not the decorations, it is clear
%  that this morphism induces a bialgebra homomorphism.
%  
%  Examples of bialgebra homomorphisms obtained in this way\ldots 123456
%\end{blanko}

\begin{blanko}{The classical \fdb: surjections versus effective trees.}
  As far as we know, the classical \fdb bialgebra of surjections is not a
  bialgebra of $P$-trees for any $P$.  There is nevertheless a close 
  relationship with the bialgebra of effective trees, which we now proceed to 
  explain.  The following construction works for any polynomial endofunctor
  without nullary operations.
  
  Since effective trees have no nullary nodes, the leaf map can be seen
  as taking values in the groupoid $\nonempt$ of non-empty finite sets.
  Pulling back along the leaf map $L:\tr\to \nonempt$ 
$$
L\upperstar :\Grpd_{/\nonempt}\to \Grpd_{/\tr},\qquad 
$$
sends $\name n:1\to\nonempt$ to the inclusion of the discrete fibre ${}_n\tr\to 
\tr$.

This yields a linear map
%%%So instead we just define by hand:
\begin{eqnarray*}
  \Q^{\pi_0 \nonempt} & \longrightarrow & \Q^{\pi_0 \tr}  \\
  A_n & \longmapsto & G_n
\\  %\end{eqnarray*}
%and extend by linearity.
%Note that we get also
%\begin{eqnarray*}
  a_n  & \longmapsto & g_n  \\
  A & \longmapsto & G .
\end{eqnarray*}
which extends to an algebra homomorphism
% 
% % and this induce
%     This map sends $A_n \in \Q^{\nonempt_0}$ to $G_n \in \Q^{\tr_0}$
%     and $A$ to $G$.
% \end{lemma}
% \begin{dem}
%     By definition, $A_n$ is the relative cardinality of the map $1 
%     \stackrel{\name{n}}\to \nonempt$, and $G_n$ is the relative 
%     cardinality of ${}_n\tr$, which is the 
%     pullback of $1 
%     \stackrel{\name{n}}\to \nonempt$ along $L$.  It remains to observe
%     that since $L$ is a discrete fibration, by Lemma{ff} the 
%     cardinality of the pullback is the pullback of the cardinality.
% \end{dem}
% 
% This linear map $L_0\upperstar:A_n\mapsto G_n$ in turn induces a multiplicative map
$$
\Phi:\fdbsymb=\Q[[A_n]]_{n\in\pi_0\nonempt}\longrightarrow
\Q[[\delta_T]]_{T\in\pi_0\tr}=\efftreebialg .
$$
\begin{lemma}\label{bialg-homo}
    The map $\Phi$ is a bialgebra homomorphism.
\end{lemma}

\begin{dem}
    We already noted that $\Phi$ preserves the Green functions.
    Now 
    \begin{eqnarray*}
    (\Phi\otimes\Phi)( \Delta(A)) &=& (\Phi\otimes\Phi)( \sum_n A^n 
    \otimes a_n) \\ &=&
    \sum_n (\Phi A)^n \otimes \Phi(a_n) \\ &=& \sum_n G^n \otimes g_n 
    \\
    &=& 
    \Delta(\Phi(A)).
    \end{eqnarray*}
    It remains to recall that the comultiplication in $\fdbsymb$ is 
    determined by the comultiplication of the Green function.
\end{dem}
% $$
% \Phi(A_n)=\norm{{}_n\tr}_\tr=\sum_{T\in({}_n\tr)_0}\delta_T.
% $$
% We claim this map preserves the Green functions,
% $$\Phi(A)=\sum\Phi(A_n)/n!=\sum G_n/n!=G.$$

% We claim also that it is in fact a bialgebra homomorphism. 
% By the \fdb formulae we have
% $$
% \Phi:\Delta(A)=\sum A^n\otimes a_n\longmapsto \Delta(G)=\sum G^n\otimes g_n
% $$
% But all comultiplication is determined by that on the Green function $A$, 
% so $\Phi$ is comultiplicative \dots

% To appreciate the amount of combinatorics hidden in these arguments,
% it is rewarding to work out the comultiplicativity of $\Phi$ by hand.
% We provide in the Appendix a direct combinatorial proof that $(\Phi\otimes\Phi)(\Delta A_4) = 
% \Delta G_4$,
% but for simplicity we work with stable trees instead of effective trees.
\end{blanko}

% TeX-master: GaKoTo1-final.tex

\begin{blanko}{Trees of graphs.}\label{trees-of-graphs}
  This final example is our main motivation for studying $P$-trees: forthcoming
  work of the second author~\cite{Kock:graphs-and-trees} shows that (nestings
  of) Feynman graphs for a given quantum field theory can be considered as
  $P$-trees for a suitable finitary polynomial endofunctor $P$ defined over
  groupoids.  We briefly describe this polynomial endofunctor and its relation
  with graphs.
   
  The relationship between trees and graphs in the Connes-Kreimer Hopf algebras   is that the trees encode nestings of
  graphs.  In the following figure,
  
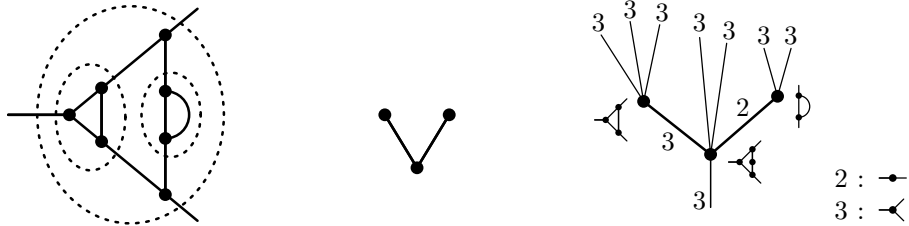
\begin{center}
  \begin{texdraw}
  \bsegment
	\linewd 1
      \move (-3 0) \lvec (20 0) \Onedot \lvec (68 40)
      \move (56 30) \Onedot \lvec (56 -30) \Onedot
      \move (32 10) \Onedot \lvec (32 -10) \Onedot
      \move (20 0) \lvec (68 -40)

      \move (56 9) \Onedot
      \move (56 -9) \Onedot
      \move (56 0) \larc r:9 sd:-90 ed:90

      \red
      \lpatt (1 3)
      \move (28 -1) \freeEllipsis{13}{20}{0}
      \move (58 0) \freeEllipsis{11}{16}{0}
      \move (43 0) \freeEllipsis{35}{41}{0}

  \esegment

  \move(150 -20)
  \bsegment
  	\linewd 1
  \move (0 0) \Onedot \lvec (-12 20) \Onedot
  \move (0 0) \lvec (12 20) \Onedot
  \esegment
  
\move (260 -35)

  \bsegment
  
    \move (0 0) 
    \lvec (0 20) \Onedot \lvec (-4 64) 
    \move (0 20) \lvec (6 62)
    \move (0 20) \linewd 1 
    \lvec (25 42) \linewd 0.5
    \Onedot \lvec (20 60) \move (25 42) \lvec (30 60)

    \move (0 20) \linewd 1 \lvec (-25 40) \linewd 0.5
 
    \bsegment 
      \move (0 0) \Onedot \lvec (-16 25) \move (0 0) \lvec (-5 30)
      \move (0 0) \lvec (6 29)
    \esegment

\move (6 17) \trekant 
\setunitscale{0.8} \rmove (12 0) \smalldot \setunitscale{1}
\move (-44 33) \trekant
\move (33 38) \tokant

\htext (-4 3){{\footnotesize $3$}}
\htext (-16 25){{\footnotesize $3$}}
\htext (-4 70){{\footnotesize $3$}}
\htext (7 67){{\footnotesize $3$}}

\htext (-42 70){{\footnotesize $3$}}
\htext (-30 75){{\footnotesize $3$}}
\htext (-18 74){{\footnotesize $3$}}

\htext (20 65){{\footnotesize $3$}}
\htext (30 65){{\footnotesize $3$}}
\htext (12 38){{\footnotesize $2$}}

\move (52 11)
\bsegment
\htext (0 0){{\footnotesize $2$ :}} \move (10 0) \tovert
\htext (0 -12){{\footnotesize $3$ :}} \move (10 -12) \trevert
\esegment
\esegment

\end{texdraw}
\end{center}
the small combinatorial tree in the middle expresses the nesting of
1PI subgraphs on the left. It is clear that such combinatorial trees
do not capture anything related to symmetries of graphs.
For this, fancier trees are needed,
as partially indicated on the right.
First of all, each node in the tree should be decorated by the 1PI graph
it corresponds to in the nesting~\cite{Bergbauer-Kreimer:0506190}, and second,
to allow an operadic interpretation,
the tree should have leaves (input
slots) corresponding to the vertices of the graph. 
Just as vertices of graphs serve as insertion points, the leaves of
a tree serve as input slots for grafting.

The decorated tree should be 
regarded as a
recipe for reconstructing the graph by inserting the decorating graphs into
the vertices of the graphs of parent nodes.
The numbers on the edges
indicate the  type constraint of each substitution: the outer interface of
a graph must match the local interface of the vertex it is substituted into.
But the type constraints on the tree decoration are not enough to reconstruct the
graph, because for example the small graph 
\raisebox{-5pt}{\begin{texdraw}\trekant\end{texdraw}} decorating
the left-hand node could be substituted into various different vertices of the 
graph
\raisebox{-5pt}{\begin{texdraw}
  \trekant  \setunitscale{0.8} \rmove (12 0) \smalldot \setunitscale{1}
\end{texdraw}}.
  The solution found in \cite{Kock:graphs-and-trees}, which 
  draws from insights
  from higher category theory~\cite{Kock-Joyal-Batanin-Mascari:0706},
  is to consider $P$-trees,
  for $P$ a certain polynomial endofunctor over groupoids, which depends on the
  theory.

  To match the figures above, we consider a theory in which there are
  two interaction labels \raisebox{2pt}{\begin{texdraw} \tovert \end{texdraw}} and
  \raisebox{-1pt}{\begin{texdraw} \trevert \end{texdraw}}\ ; let $I$ denote
  the groupoid of all such one-vertex graphs.  Let $B$ denote the groupoid
  of all connected 1PI graphs of the theory that are primitive in the Hopf 
  algebra of graphs.  Finally let $E$ denote the groupoid of such graphs with a marked
  vertex.  The polynomial endofunctor $P$ is now given by the diagram
\begin{equation*}
\xymatrix{
    I & \ar[l]_s  E  \ar[r]^p & B  \ar[r]^t & I ,
}
\end{equation*}
  where the map $s$ returns the one-vertex subgraph at
  the mark, $p$ forgets the mark, and $t$ returns the residue of the
  graph, i.e.~the graph obtained by contracting everything to a point, but
  keeping the external lines.  A $P$-tree is hence a diagram 
    \begin{equation*}\label{qtree}
  \xymatrix @!C=16pt {
  A\ar[d]_\alpha \ar@{}[dr]|{\Leftarrow}& M\ar[l]  
  \ar@{}[dr]|{\Rightarrow}\drpullback\ar[r] \ar[d]& N \ar[d] \ar[r] \ar@{}[dr]|{\Rightarrow}
  &A\ar[d]^\alpha \\
  I & \ar[l] E  \ar[r] & B \ar[r]&I \,, \\
  }\end{equation*} with specified $2$-cells, in which the first row is a tree in
  the sense of~\ref{polytree-def}.  These $2$-cells carry much of the structure:
  for example the $2$-cell on the right says that the 1PI graph decorating a
  given node must have the same residue as the decoration of the outgoing edge
  of the node --- or more precisely, and more realistically: an isomorphism is
  specified (it's a bijection between external lines of one-vertex graphs).
  Similarly, the left-hand $2$-cell specifies for each
  node-with-a-marked-incoming-edge $x'\in M$, an isomorphism between the
  one-vertex graph decorating that edge and the marked vertex of the graph
  decorating the marked node $x'$.  Hence the structure of a $P$-tree is a
  complete recipe not only for which graphs should be substituted into which
  vertices, but also {\em how}: specific bijections prescribe which external
  lines should be identified with which lines in the receiving graph.  
  
  More precisely, the result of \cite{Kock:graphs-and-trees} states an
  equivalence of groupoids.  In particular a $P$-tree has the same symmetry
  group as the graph (with its nesting) that it encodes, so that the Green
  functions match up, and in the end the \fdb formula in the bialgebra of
  $P$-trees can be transported to a certain bialgebra of graphs.
  However, this bialgebra of graphs is not quite the same as the standard
  Connes--Kreimer Hopf algebra of graphs, and we are not yet able to derive van
  Suijlekom's \fdb formula from our general framework.  The main dificulty lies
  in getting a purely operadic encoding of the line insertions, allowed in the
  Connes--Kreimer Hopf algebra but not in our bialgebra of graphs.  This issue is
  closely related to the renormalisation factors $1/\sqrt{G_e}$ 
  mentioned in the Introduction.

%   which draws from insights
%   from higher category theory~\cite{Kock-Joyal-Batanin-Mascari:0706}, states
%   that there is an equivalence of groupoids between the groupoid of
%   $P$-trees and a certain groupoid of graph nestings.  In particular, the
%   symmetries of a given graph nesting can be read off the corresponding
%   decorated tree and vice versa.
% 
% (It should be stressed that the use of groupoids as coefficients
%   is crucial for getting the decorations that make this 
%   correspondence to work.  In fact, a tree decorated in groupoids may have
%   more symmetries than the underlying tree.  For example, the graph $\Gamma = \ 
%   $
%   \raisebox{0pt}{\begin{texdraw}\rundtokant\end{texdraw}}
%   is 1PI and, as a trivial nesting it corresponds to the tree
%   \raisebox{-2pt}{\begin{texdraw}
% \bsegment \setunitscale{0.8}
% \move (0 -5) \lvec (0 2) \smalldot \lvec (-4 9) \move (0 2)  \lvec (4 9)
% \esegment
% \end{texdraw}} decorated with $\Gamma$ at the node, $3$ at the leaves, and $2$
% at the root.  More formally it is of course a diagram like \eqref{qtree}.
% It is straightforward to check that this $P$-tree has a symmetry
% group of order $4$, just as the graph $\Gamma$, whereas the underlying tree
% clearly has a symmetry group of order $2$.  
\end{blanko}

\bibliographystyle{plain}
\bibliography{GKT1}

\end{document}